\documentclass[reqno]{amsart}

\usepackage{amsfonts}
\usepackage{color}
\usepackage{amssymb}
\usepackage{amsmath}
\usepackage{amsaddr}
\usepackage{amsaddr}

\newtheorem{theorem}{Theorem}[section]

\newtheorem{conjecture}[theorem]{Conjecture}
\newtheorem{corollary}[theorem]{Corollary}

\newtheorem{definition}[theorem]{Definition}

\newtheorem{lemma}[theorem]{Lemma}

\newtheorem{proposition}[theorem]{Proposition}

\theoremstyle{remark}
\newtheorem{remark}[theorem]{Remark}

\numberwithin{equation}{section}

\newcommand{\ep}{\varepsilon}
\newcommand{\s}[1]{{\mathcal #1}}

\begin{document}

\title[Damped Wave Equations with Dynamic Boundary Conditions]{Attractors for Strongly Damped Wave Equations with Nonlinear Hyperbolic Dynamic Boundary Conditions}

\author[P. J. Graber and J. L. Shomberg]{P. Jameson Graber$^1$ and Joseph L. Shomberg$^2$}



\address{$^1$P. Jameson Graber, Naveen Jindal School of Management, The University of Texas at Dallas, 800 West Campbell Road, SM30, Richardson, Texas 75080, \\ {\tt{pjg140130@utdallas.edu}}}

\address{$^2$Joseph L. Shomberg, Department of Mathematics and Computer Science, Providence College, 1 Cunningham Square, Providence, Rhode Island 02918, USA, \\ {\tt{jshomber@providence.edu}}}

\date{\today}

\thanks{$^1$Research supported by NSF grant DMS-1303775.}

\begin{abstract}
We establish the well-posedness of a strongly damped semilinear wave equation equipped with nonlinear hyperbolic dynamic boundary conditions.
Results are carried out with the presence of a parameter distinguishing whether the underlying operator is analytic, $\alpha>0$, or only of Gevrey class, $\alpha=0$. 
We establish the existence of a global attractor for each $\alpha\in[0,1],$ and we show that the family of global attractors is upper-semicontinuous as $\alpha\rightarrow0.$
Furthermore, for each $\alpha\in[0,1]$, we show the existence of a weak exponential attractor.
A weak exponential attractor is a finite dimensional compact set in the weak topology of the phase space.
This result insures the corresponding global attractor also possess finite fractal dimension in the weak topology; moreover, the dimension is independent of the perturbation parameter $\alpha$.
In both settings, attractors are found under minimal assumptions on the nonlinear terms.
\end{abstract}

\maketitle
\tableofcontents

\section{Introduction}

Let $\Omega $ be a bounded domain in $\mathbb{R}^{3}$ with boundary $\Gamma:=\partial \Omega $ of class $C^{2}$. 
We consider the semilinear strongly damped wave equation, 
\begin{equation}
u_{tt} - \omega\Delta u_t + u_t - \Delta u + u + f(u) = 0 \quad \text{in} \quad (0,\infty) \times \Omega,  \label{pde}
\end{equation}
where $0<\omega\le1$ represents the diffusivity of the momentum.
The equation is endowed with the dynamic boundary condition, with $0\le\alpha\le 1,$
\begin{equation}
u_{tt} + \partial_{\mathbf{n}}(u+\omega u_t) - \alpha\omega\Delta_\Gamma u_t + u_t - \Delta_\Gamma u + u + g(u) = 0 \quad \text{on} \quad (0,\infty )\times \Gamma,  \label{bc}
\end{equation}
and with the initial conditions, 
\begin{equation}
u(0,x) = u_{0}(x), \quad u_{t}(0,x) = u_{1}(x) \quad \text{at} \quad \{0\}\times \Omega,  \label{ic-1}
\end{equation}
and
\begin{equation}
u_{\mid\Gamma}(0,x)=\gamma_0(x), \quad u_{t\mid\Gamma}(0,x)=\gamma_1(x) \quad \text{at} \quad \{0\}\times \Gamma.  \label{ic-2}
\end{equation}

For the nonlinear terms, we assume $f,g\in C(\mathbb{R})$ satisfies the sign conditions, 
\begin{align}
\liminf_{|s|\rightarrow \infty }\frac{f(s)}{s}>-1,  \label{assm-1} \\
\liminf_{|s|\rightarrow \infty }\frac{g(s)}{s}>-1,  \label{assm-2}
\end{align}
and we assume the growth assumptions hold, for all $r,s\in \mathbb{R}$, 
\begin{align}
|f(r)-f(s)| & \le \ell_1|r-s|\left(1+|r|^{2} + |s|^{2} \right),  \label{assm-3} \\
|g(r)-g(s)| & \le \ell_2|r-s|\left(1+|r|^{\rho-1}+|s|^{\rho-1}\right),  \label{assm-4}
\end{align}
for some positive constants $\ell_1,\ell_2$, and $2 \le \rho < \infty$.
We will refer to equations (\ref{pde})--(\ref{ic-2}) under assumptions (\ref{assm-1})--(\ref{assm-4}) as Problem {\textbf{P}}$_\alpha$, for $\alpha\in[0,1]$.

Problem {\textbf{P}}$_\alpha$ draws motivation from viscoelastic material; i.e., physical phenomena that exhibit both elasticity and viscosity when undergoing deformation. 
In models that approximate the behavior of non-Hookean materials under high strains, the term $-\Delta u_t$ not only indicates that the stress is proportional to the strain, but, in addition, the term communicates that the stress is proportional to the strain {\em{rate}}.
Thus, such terms appear when modeling viscoelastic materials such as Kelvin--Voigt type materials (cf. e.g. \cite[Section 4.9.2]{Hale88} and \cite[Section 13.10]{Meyers&Chawla_2008}).
In the present work, we allow the term $-\omega\Delta u_t$, $0<\omega\le1$, to appear in the sense of a strong damping perturbation to the (weakly) damped semilinear wave equation. 
Hence, equation (\ref{pde}) contains the perturbed (homogeneous) sine-Gordon equation,
\[
u_{tt} - \omega\Delta u_t + u_t - \Delta u + \sin u = 0, 
\]
used in modeling the evolution of the current $u$ in a Josephson junction (cf. \cite{Hale88,LSC82}).
Moreover, Problem {\textbf{P}}$_\alpha$ may be used to describe the perturbed Klein--Gordon wave equation appearing in quantum mechanics (cf. \cite[Section IV.3]{Temam88}),
\[
u_{tt} - \omega\Delta u_t + u_t - \Delta u + |u|^{\gamma-1}u = 0,
\]
for $1\le\gamma\le3.$

The damped wave equation,
\begin{equation}  \label{stda}
u_{tt}+Au_t+Bu=F(u),
\end{equation}
has been the topic of several important works, of which we will only mention a few. 
For standard Dirichlet--Laplacian operators, $A=-\omega\Delta$ and $B=-\Delta$, and any potential $F:D(-\Delta)=H^2(\Omega)\cap H^1_0(\Omega)\rightarrow L^2(\Omega)$ locally Lipschitz continuous, \cite{Webb80} established the global existence of strong solutions.
No differentiability assumption on $F$ is needed for the result; indeed, the result can be attributed to the local existence theorem for analytic semigroups (cf. e.g. \cite{Pazy83}). 
Carvalho and Cholewa \cite{Carvalho_Cholewa_02} show the existence of local weak solutions (in $H^1_0(\Omega)\times L^2(\Omega)$) to \eqref{stda} with $A^\theta=\omega(-\Delta)^{\theta}$, $\theta\in[\frac{1}{2},1]$, $B=-\Delta$, and $F\in C(\mathbb{R})$ satisfying
\begin{align}
|F(r)-F(s)| & \le \ell_1|r-s|\left(1+|r|^{\rho-1}+|s|^{\rho-1}\right),
\end{align}
where $1\le\rho\le\frac{n+2}{n-2}$ for $\Omega\subset\mathbb{R}^n$ bounded and smooth.
For course, of particular importance is the critical nonlinearity $\rho=5.$
To that end, the same authors prove in \cite{Carvalho_Cholewa_02-2} the global well-posedness for the problem associated to \eqref{stda}.
They also show the existence of compact global (universal) attractor in $H^1_0(\Omega)\times L^2(\Omega)$ with the aid of the dissipation assumption \eqref{assm-1}.
For this result, the case $\theta=1$ is important because the associated operator,
\begin{equation}  \label{st-op}
\mathcal{A^\theta}:=\begin{pmatrix} 0 & -I \\ B & A^\theta \end{pmatrix},
\end{equation}
does not possesses a compact resolvent (cf. \cite[Proposition 1]{Carvalho_Cholewa_02}). 
It is certainly worth mentioning the work \cite{Chen&Triggiani90} who show that the operator \eqref{st-op} is of Gevrey class for $\theta\in(0,\frac{1}{2}).$
Pata and Squassina prove in \cite{Pata&Squassina05} that the subcritical problem with $\theta=1$ admits an exponential attractor of optimal regularity in the standard energy phase space.
In \cite{Pata&Zelik06}, the authors Pata and Zelik show the problem with critical and supercritical nonlinearities also admit global attractors with optimal regularity.
In the present work, we do not consider the critical case, rather, for the existence of a local (mild) solution will rely on classical semigroup theory.
Additionally, our results are presented for the $\theta=1$ setting. 

Much of the present literature on the strongly damped wave equation contains only the case of the Dirichlet boundary condition. 
Of course, sometimes only trivial modification are required to reproduce the results for Neumann, Robin, or periodic boundary conditions. 
However, it is becoming increasingly apparent that the importance of so-called {\em{dynamic}} boundary conditions be considered and developed in future mathematical models. 
The damped wave equation with dissipation appearing (in the sense of fractional damping) on the boundary appears in \cite{Wu&Zheng06}.
Since this result appeared, it has become physically relevant and hence important to adopt dynamic boundary conditions.  
A source of special emphasis on this front appears quite naturally in the analysis of process of spinodial decomposition, relevant to the Cahn--Hilliard equations.
We quote \cite{GMS2010}:
\begin{quote}
In most works, the equations are endowed with Neumann boundary conditions for both [unknowns] $u$ and $w$ (which means that the interface is orthogonal to the boundary and that there is no mass flux at the boundary) or with periodic boundary conditions. Now, recently, physicists have introduced the so-called dynamic boundary conditions, in the sense that the kinetics, i.e., $\partial_t u$, appears explicitly in the boundary conditions, in order to account for the interaction of the components with the walls for a confined system.
\end{quote}
The dynamic boundary condition present in \eqref{bc} is of hyperbolic type.
A related hyperbolic boundary condition---though not exhibiting surface diffusion---appears in \cite{Graber_Said-Houari_2012}. 
The present Problem {\textbf{P}}$_\alpha$, $\alpha\in[0,1]$, is the strongly damped perturbation of the (weakly) damped semilinear wave equation.
It describes the dynamics of a wave in a bounded domain under the influence of dissipative effects. 
In our model, the importance of the momentum at the boundary of the domain is not neglected; hence, an important dynamical feature of our model is the effect of the dynamical flux present on the boundary; i.e., the $\partial_{\mathbf{n}}u_t$ term which describes the evolution of the surface (tangential) gradient of the velocity component on $\Gamma$.
For example, Problem {\textbf{P}}$_\alpha$, $\alpha\in[0,1]$, may describe a gas experiencing irrotational forces from a rest state in a domain $\Omega$. 
The surface $\Gamma$, now governed by its own wave equation, acts as a locally reacting dissipation mechanism in response to excess pressure in $\Omega$. 
Hence, (\ref{bc}) describe $\Gamma$ as a so-called locally reactive surface. 

The main results in this paper are:

\begin{itemize}

\item For each $\alpha\in[0,1],$ we establish the existence and uniqueness of global mild solutions under only minimal assumptions on the nonlinear terms.

\item The global mild solutions generate a locally Lipschitz continuous semiflow on the standard energy phase space.

\item For each $\alpha\in[0,1]$, the semiflow admits a bounded absorbing set, bounded in the phase space independent of the parameter $\alpha$.

\item The semiflow admits a family of global attractors for each $\alpha\in[0,1].$
The required asymptotic compactness for the semiflow is established using a suitable $\alpha$-contraction argument.
We show that the family of global attractors is upper-semicontinuous with respect to the parameter $\alpha$ as $\alpha\rightarrow0.$

\item Finally, for each $\alpha\in[0,1]$, the existence of a so-called weak exponential attractor is proven.
This result guarantees the finite (fractal) dimension of the global attractors in the weak topology. The dimension is uniform in $\alpha.$

\end{itemize}

{\textbf{Notation and conventions.}} We take the opportunity here to introduce some notations and conventions that are used throughout the paper. 
Norms in the associated space are clearly denoted $\|\cdot\|_{B}$ where $B$ is the corresponding Banach space. 
We use the notation $\langle \cdot ,\cdot \rangle_{H}$ to denote the inner product on the Hilbert space $H$.
In many calculations, functional notation indicating dependence on the variable $t$ is dropped; for example, we will write $u$ in place of $u(t)$. 
Throughout the paper, $C>0$ will denote a \emph{generic} constant, while $Q:\mathbb{R}_{+}\rightarrow \mathbb{R}_{+}$ will denote a \emph{generic} increasing function. 
All these quantities may depend on various structural parameters, however, unless explicitly stated, they are \emph{independent} of the perturbation parameter $\alpha$.
Let $\lambda_\Omega>0$ denote the best constant satisfying the Sobolev/Poincar\'{e} inequality in $\Omega$,
\begin{equation}  \label{Poincare}
\lambda_\Omega \int_\Omega u^2 dx \le \int_\Omega\left( |\nabla u|^2+u^2 \right) dx.
\end{equation}
We will also rely on the Laplace--Beltrami operator $-\Delta_\Gamma$ on the surface $\Gamma.$ 
This operator is positive definite and self-adjoint on $L^2(\Gamma)$ with domain $D(-\Delta_\Gamma)$.
The Sobolev spaces $H^s(\Gamma)$, for $s\in\mathbb{R}$, may be defined as $H^s(\Gamma)=D((-\Delta_\Gamma)^{s/2})$ when endowed with the norm whose square is given by, for all $u\in H^s(\Gamma)$,
\begin{equation}  \label{LB-norm}
\|u\|^2_{H^s(\Gamma)} := \|u\|^2_{L^2(\Gamma)} + \left\|(-\Delta_\Gamma)^{s/2}u\right\|^2_{L^2(\Gamma)}.
\end{equation}
On the boundary, let $\lambda_\Gamma>0$ denote the best constant satisfying the Sobolev/Poincar\'{e} inequality on $\Gamma$,
\begin{equation}  \label{bndry-Poincare}
\lambda_\Gamma \int_\Gamma u^2 d\sigma \le \int_\Gamma\left( |\nabla_\Gamma u|^2+u^2 \right) d\sigma.
\end{equation}

Throughout the paper, the reader should be mindful that the results contained here belong to two classes corresponding to the ``analytic'' Problem {\textbf{P}}$_\alpha$, where $\alpha\in(0,1]$, and to the ``almost analytic'' or Gevrey Problem {\textbf{P}}$_0$, when $\alpha=0$.
As for the plan of the paper, in Section \ref{s:functional}, we review the functional setting and framework for the abstract model problem; in particular, the standard energy phase space is introduced and the appropriate semigroups associated with Problem {\textbf{P}}$_\alpha$ and Problem {\textbf{P}}$_0$ are discussed. 
In Section \ref{s:sf} we determine various properties of the solutions to Problem \textbf{P}$_\alpha$ and Problem {\textbf{P}}$_0$. 
The well-posedness of Problem {\textbf{P}}$_\alpha$ is obtained by virtue of the analyticity of the underlying operator; i.e., it is the infinitesimal generator of an analytic semigroup, whereas the well-posedness of Problem {\textbf{P}}$_0$ relies on the fact that the underlying operator is known to generate a $C_0$-semigroup of contractions of Gevrey class $\delta$, for $\delta>2.$ 
The above assumptions on the nonlinear terms define a locally Lipschitz functional on the standard energy phase space. 
Together, semigroup theory provides the existence of local mild solutions which are readily found to be globally defined. 
Uniqueness of the solutions follows from a continuous dependence estimate utilizing sharp Sobolev embeddings.
The mild solutions generate a locally Lipschitz continuous semiflow, uniformly in $t$ on compact intervals, on the standard energy phase space. 
In Section \ref{s:dissipativity} we prove the existence of a bounded absorbing set admitted by the semiflow. 
This result holds for all problems $\alpha\in[0,1]$ with a bound independent of $\alpha$.
In order to establish the existence of a global attractor for each $\alpha\in[0,1],$ it suffices to prove the associated semiflows are precompact.
As in \cite[Proposition 1]{Carvalho_Cholewa_02} when $\theta=1$ (which also occurs in our case), the semigroup of solution operators is {\em{not}} compact; which in turn means we cannot rely on the regularizing effects of the solution operators to obtain the required (pre)compactness.
Instead, to obtain the precompactness of the solution operators, we rely on the method of $\alpha$-contractions (cf. \cite{Chow_Hale_1974} and the references therein).
Finally, in Section \ref{s:exp-attr} we prove the existence of weak exponential attractors which are compact in the weak topology and bounded in the standard phase space.
Using this result, we are able to show the global attractors possess finite fractal dimension only in the weak topology.
The final Section \ref{s:conclusions} contains some remarks and notes for future research, as well as some observations and conjectures on the characterization of the domain for the fractional powers of the operator $A_\alpha$ when $\alpha\in(0,1]$. 

\section{Functional setting}  \label{s:functional}

We begin with the consequences of the assumptions made on the nonlinear terms.
From assumption (\ref{assm-1}) and the definition of the $H^1(\Omega)$ norm, it follows that, for some constants $\mu_1 \in (0,1]$ and $c_1=c_1(f,|\Omega|)\ge 0$, and for all $\xi \in H^1(\Omega)$,
\begin{align}
\langle f(\xi), \xi \rangle_{L^2(\Omega)} & \ge -(1 - \mu_1) \|\xi\|^2_{L^2(\Omega)} - c_1  \notag \\ 
& \ge -(1 - \mu_1) \|\xi\|^2_{H^1(\Omega)} - c_1.  \label{cons-3}
\end{align}
For some constant $c_2=c_2(f,|\Omega |)\ge 0$, and for all $\xi \in L^2(\Omega)$, 
\begin{align} \int_\Omega F(\xi) dx & \ge
-\frac{1-\mu_1}{2}\|\xi\|^2_{L^2(\Omega)} - c_2  \notag \\ 
& \ge -\frac{1-\mu_1}{2}\|\xi\|^2_{H^1(\Omega)} - c_2,  \label{cons-4}
\end{align}
where $F(s)=\int_0^s f(\xi) d\xi$.
Notice that (\ref{assm-3}) and (\ref{assm-4}) imply, when we fix $s=0$, that for all $r\in\mathbb{R}$, 
\begin{align}
|f(r)| & \le \ell_1 \left( |r| + |r|^{3} \right) + |f(0)|,  \label{afix-1} \\
|g(r)| & \le \ell_2 \left( |r| + |r|^{\rho} \right) + |g(0)|.  \label{afix-2}
\end{align}
Together (\ref{assm-1}), (\ref{Poincare}), (\ref{afix-1}), and the continuous embedding $H^1(\Omega) \hookrightarrow L^6(\Omega)$ give the upper-bounds, for all $\xi\in H^1(\Omega),$
\begin{align} 
\int_\Omega F(\xi) dx & \le \ell_1 \left( \|\xi\|^2_{L^2(\Omega)} + \|\xi\|^6_{L^6(\Omega)} \right) + |f(0)|\|\xi\|_{L^1(\Omega)} \notag \\
& \le C \left( \|\xi\|^2_{H^1(\Omega)} + \|\xi\|^6_{H^1(\Omega)} + \|\xi\|_{H^1(\Omega)} \right),  \label{cons-5}
\end{align}
where $C=C(\ell_1,\lambda_\Omega,\Omega,f).$
Reflecting the above estimates, due to (\ref{assm-2}) and (\ref{LB-norm}), there are constants $\mu_2\in(0,1]$, $c_3, c_4\ge0$ such that for all $\xi\in L^2(\Gamma),$
\begin{align}  \label{cons-6}
\langle g(\xi),\xi \rangle_{L^2(\Gamma)} & \ge -(1-\mu_2)\|\xi\|^2_{L^2(\Gamma)} - c_3  \notag \\ 
& \ge -(1-\mu_2)\|\xi\|^2_{H^1(\Gamma)} - c_3,
\end{align}
\begin{equation}  \label{cons-7}
\int_\Gamma G(\xi) d\sigma \ge -\frac{1-\mu_2}{2}\|\xi\|^2_{L^2(\Gamma)} - c_4,
\end{equation}
where $G(s)=\int_0^s g(\xi) d\xi$ and $d\sigma$ represents the natural surface measure on $\Gamma$, and with (\ref{assm-2}), (\ref{bndry-Poincare}), (\ref{afix-2}), and the embedding $H^1(\Gamma)\hookrightarrow L^p(\Gamma)$, we also find, for all $\xi\in H^1(\Gamma),$ 
\begin{align} 
\int_\Omega G(\xi) d\sigma & \le \ell_2 \left( \|\xi\|^2_{L^2(\Gamma)} + \|\xi\|^{\rho+1}_{L^{\rho+1}(\Gamma)} \right) + |g(0)|\|\xi\|_{L^1(\Gamma)}  \notag \\
& \le C \left( \|\xi\|^2_{H^1(\Gamma)} + \|\xi\|^{\rho+1}_{H^1(\Gamma)} + \|\xi\|_{H^1(\Gamma)} \right),  \label{cons-8}
\end{align}
where here $C=C(\ell_2,\lambda_\Gamma,\Gamma,g).$

The ``standard energy'' phase space and abstract formulation for Problem {\textbf{P}}$_\alpha$, $\alpha\in[0,1]$, are now given. 
Let 
\[
\mathcal{H}_0 := H^1(\Omega) \times L^2(\Omega) \times H^1(\Gamma) \times L^2(\Gamma).
\]
The space $\mathcal{H}_0$ is Hilbert with the norm whose square is given by, for $\zeta=(u,v,\gamma,\delta)\in\mathcal{H}_0$,
\begin{align}
\|\zeta\|^2_{\mathcal{H}_0} & := \|u\|^2_{H^1(\Omega)} + \|v\|^2_{L^2(\Omega)} + \|\gamma\|^2_{H^1(\Gamma)} + \|\delta\|^2_{L^2(\Gamma)}   \notag \\ 
& = \left( \|\nabla u\|^2_{L^2(\Omega)} + \|u\|^2_{L^2(\Omega)} \right) + \|v\|^2_{L^2(\Omega)} + \left( \|\nabla_\Gamma\gamma\|^2_{L^2(\Gamma)} + \|\gamma\|^2_{L^2(\Gamma)} \right) + \|\delta\|^2_{L^2(\Gamma)}.   \notag
\end{align}
Before we continue, let us now recall that the Dirichlet trace map, $tr_{D}:C^{\infty}\left( \overline{\Omega }\right) \rightarrow C^{\infty }\left( \Gamma\right),$ defined by $tr_{D}\left(u\right)=u_{\mid\Gamma},$
extends to a linear continuous operator $tr_{D}:H^{r}\left( \Omega \right) \rightarrow H^{r-1/2}\left( \Gamma \right) ,$ for all $r>\frac{1}{2}$, which is surjective for $\frac{1}{2}<r<\frac{3}{2}.$ 
This map also possesses a bounded right inverse $tr_{D}^{-1} : H^{r-1/2}\left( \Gamma \right) \rightarrow H^{r}\left( \Omega \right) $ such that $tr_{D}(tr_{D}^{-1}\psi)=\psi,$ for any $\psi \in H^{r-1/2}\left(\Gamma \right)$.

We find it worthwhile to repeat \cite[Remark 1.1]{Graber&Lasiecka14}:

\begin{remark}
In the space $\mathcal{H}_0$, the trace of $u_t(0)\in L^2(\Omega)$ is not well-defined in $L^2(\Gamma).$
This also means that we cannot identify the second and fourth components of some $\zeta_0=(u_0,v_0,\gamma_0,\delta_0)\in\mathcal{H}_0$ through the trace. 
However, we will see that along trajectories of Problem {\textbf{P}}$_\alpha$, $\alpha\in[0,1],$ when $t>0$ the identification $tr_D(u_t(t))=u_{t\mid\Gamma}(t)$ is allowed. 
The instantaneous regularization of the solutions to Problem {\textbf{P}}$_\alpha$, $\alpha\in[0,1]$, guarantees the trace is well-defined.
\end{remark}

For any $\alpha\in[0,1]$, $\omega\in(0,1]$, let
\begin{align}
D(A_{\alpha}):=\left\{ \zeta=(u,v,\gamma,\delta)\in \mathcal{H}_0 : v\in H^1(\Omega), ~v_{\mid\Gamma}\in H^1(\Gamma), ~\Delta(u+\omega v)\in L^2(\Omega), \right.  \notag \\ 
\left. ~\partial_{\mathbf{n}}(u+\omega v)_{\mid\Gamma}-\Delta_\Gamma(u_{\mid\Gamma}+\alpha\omega v_{\mid\Gamma})\in L^2(\Gamma), ~\gamma=u_{\mid\Gamma}, ~\delta=v_{\mid\Gamma} ~\text{on}~\Gamma \right\},  \label{domain}
\end{align}
and define the linear unbounded operator $A_{\alpha}:D(A_{\alpha})\subset\mathcal{H}_0\rightarrow\mathcal{H}_0$ by
\begin{equation}  \label{operator}
A_{\alpha}:=\begin{pmatrix} 0 & I & 0 & 0 \\ \Delta-I & \omega\Delta-I & 0 & 0 \\ 0 & 0 & 0 & I \\ -\partial_{\mathbf{n}} & -\omega\partial_{\mathbf{n}} & \Delta_\Gamma-I & \alpha\omega\Delta_\Gamma-I \end{pmatrix}.
\end{equation}

\begin{lemma}  \label{t:operator-A}
For each $\alpha\in[0,1]$, the operator $A_{\alpha}:D(A_{\alpha})\subset\mathcal{H}_0\rightarrow\mathcal{H}_0$ given in (\ref{domain})--(\ref{operator}) is closed, densely defined, and is the infinitesimal generator of a $C_0$-semigroup of contractions in $\mathcal{H}_0.$
We denote this semigroup by $e^{A_{0}t}$ when $\alpha=0$, or by $e^{A_{\alpha}t}$ for $\alpha\in(0,1]$.

\begin{enumerate}

\item When $\alpha=0$, the semigroup $e^{A_{0}t}$ is of Gevrey class $\delta$ for all $\delta>2$.
More specifically, for $C$ and $R$ sufficiently large, we have the following estimates:
\begin{align}
\|A_0e^{A_0t}\|_{\mathcal{L(H}_0)} &\leq \frac{C}{t^2}, \quad \forall  t > 0,  \notag \\
\|R(i\beta,A_0)\|_{\mathcal{L(H}_0)} &\leq \frac{C}{|\beta|^{1/2}}, \quad \forall  \beta \in (-\infty,-R) \cup (R,\infty).  \notag
\end{align}

\item When $\alpha\in(0,1]$, the semigroup $e^{A_{\alpha}t}$ is analytic in $\mathcal{H}_0.$ 

\end{enumerate}

\end{lemma}

\begin{proof}
By the proof of \cite[Proposition 2.2]{Graber&Lasiecka14}, we see that, for each $\alpha\in[0,1],$ the operator $A_\alpha$
is closed and densely defined, and it generates a $C_0$-semigroup of contractions on $\mathcal{H}_0$.
Then by the proof of \cite[Theorem 1.5]{Graber&Lasiecka14}, we also find that the semigroup $e^{A_{0}t}$ is of Gevrey class $\delta$ for $\delta>2$, and the semigroup $e^{A_{\alpha}t}$ is analytic in $\mathcal{H}_0$. 
The only difference is that, whereas in \cite{Graber&Lasiecka14} the Poincar\'e inequality is valid so that the norm $\|u\|_{H^1(\Omega)}$ is equivalent to $\|\nabla u\|_{L^2(\Omega)}$, in our case we use the additional ``static damping" terms to get control of the full $H^1(\Omega)$ norm of $u$.
All necessary adjustments are therefore trivial.
This completes the proof.
\end{proof}

The fractional powers of $A_0$ are defined here in anticipation of Theorem \ref{t:sg-theory} below.
(Fractional powers of $A_\alpha$, $\alpha\in(0,1]$ are discussed in Section \ref{s:conclusions}.)
For any $\theta>0$, we define the fractional power of the linear operator $A_0^\theta$ as follows (cf. e.g. \cite[II.5.c, in particular, Definition 5.31]{Engel&Nagel00}): first let $\Sigma$ be an open sector in $\mathbb{C}$ such that $\mathbb{R}_+\subset\Sigma\subset\rho(A_0)$. 
Next, define 
\[
A^{-\theta}_0:=\frac{1}{2\pi i}\int_\gamma \lambda^{-\theta}R(\lambda,A_0)d\lambda,
\]
where $\gamma$ is any piecewise smooth path in $\Sigma$ connecting $\infty e^{-i\phi}$ to $\infty e^{i\phi}$, for some $\phi>0$.
Then, for $\theta>0$, the operator $A^\theta_0$ is defined as the inverse of $A_0^{-\theta}$ with $D(A_0^\theta)={\mathrm{Range}}(A_{0}^{-\theta})$.

\begin{remark}  \label{r:adjoint}
For all $\alpha\in[0,1]$, the operator $A_{\alpha}$ is dissipative on $\mathcal{H}_0.$
Indeed, for any $\zeta=(u,v,\gamma,\delta)\in\mathcal{H}_0,$
\begin{align}
\left\langle A_{\alpha}\zeta,\zeta \right\rangle_{\mathcal{H}_0} & = -\omega\|\nabla v\|^2_{L^2(\Omega)} - \|v\|^2_{L^2(\Omega)} - \alpha\omega\|\nabla_\Gamma\delta\|^2_{L^2(\Gamma)} - \|\delta\|_{L^2(\Gamma)}.  \notag
\end{align}
Moreover, if we define an operator $A^*_{\alpha}$ with domain
\begin{align}
D(A^*_{\alpha}) := \left\{ \theta=(\chi,\psi,\phi,\xi)\in \mathcal{H}_0 : \psi\in H^1(\Omega), ~\psi_{\mid\Gamma}\in H^1(\Gamma), ~\Delta(\chi-\omega\psi)\in L^2(\Omega), \right.  \notag \\ 
\left. ~\partial_{\mathbf{n}}(\chi-\omega\psi)_{\mid\Gamma}-\Delta_\Gamma(\chi_{\mid\Gamma}-\alpha\omega\psi_{\mid\Gamma})\in L^2(\Gamma), ~\phi=\chi_{\mid\Gamma}, ~\xi=\psi_{\mid\Gamma} ~\text{on}~ \Gamma \right\},  \notag
\end{align}
for all $\alpha\in(0,1]$ and $\omega\in(0,1]$ by, 
\[
A^*_{\alpha}:=\begin{pmatrix} 0 & -I & 0 & 0 \\ -\Delta+I & \omega\Delta-I & 0 & 0 \\ 0 & 0 & 0 & -I \\ -\partial_{\mathbf{n}} & \omega\partial_{\mathbf{n}} & -\Delta_\Gamma+I & \alpha\omega\Delta_\Gamma-I \end{pmatrix},
\]
then there holds,
\[
\left\langle A_{\alpha}\zeta,\theta \right\rangle_{\mathcal{H}_0} - \left\langle \zeta, A^*_{\alpha}\theta \right\rangle_{\mathcal{H}_0} = \langle \partial_{\mathbf{n}}(u+\omega v),\psi-\xi \rangle_{L^2(\Gamma)} - \langle  v-\delta,\partial_{\mathbf{n}}(\chi-\omega\psi) \rangle_{L^2(\Gamma)}.
\]
Hence, in the limit $\omega=0$, we see the operator $\{A^*_{\alpha}\}_{\mid\omega=0}$ is the adjoint of the operator $\{A_{\alpha}\}_{\mid\omega=0}$ associated with the weakly damped wave equation endowed with nonlinear hyperbolic boundary conditions.
It is important to note that when any operator $A$ possesses an explicit adjoint $A^*$, then mild solutions---coming directly from the solution methods of semigroup theory---are in fact equivalent to the variational formulation of weak solutions.
For more on this, see \cite[Section 3]{Ball04}.
\end{remark}

Finally, define the map $\mathcal{F}:\mathcal{H}_0\rightarrow\mathcal{H}_0$ by 
\[
\mathcal{F}(\zeta):=\begin{pmatrix} 0 \\ -f(u) \\ 0 \\ -g(\gamma) \end{pmatrix}
\]
for all $\zeta=(u,v,\gamma,\delta)\in\mathcal{H}_0$.
Then Problem {\textbf{P}}$_\alpha$, $\alpha\in[0,1]$, may be put into the abstract form in $\mathcal{H}_0,$
\begin{equation}   \label{abstract-p}
\left\{
\begin{array}{ll} \displaystyle\frac{d\zeta}{dt} = A_{\alpha}\zeta + \mathcal{F}(\zeta) & \text{for}\ t>0, \\ 
\zeta(0)=\zeta_0, \end{array}
\right.
\end{equation}
where $\zeta=\zeta(t)=(u(t),u_t(t),u_{\mid\Gamma}(t),u_{t\mid\Gamma}(t))$ and $\zeta_0=(u_0,u_1,\gamma_0,\gamma_1)\in\mathcal{H}_0$, now where $v=u_t$ and $\delta=\gamma_t$ in the sense of distributions.

The following result will be used to obtain local mild solutions for Problem {\textbf{P}}$_\alpha$, $\alpha\in[0,1]$, in the next section.

\begin{lemma}  \label{t:operator-F}
Under the assumptions (\ref{assm-3})--(\ref{assm-4}), the map $\mathcal{F}:\mathcal{H}_0\rightarrow \mathcal{H}_0$ is locally Lipschitz continuous. 
\end{lemma}

\begin{proof}
The proof is straightforward.
Let us first assume (\ref{assm-3})--(\ref{assm-4}) hold. 
By the definition of $\mathcal{F}$ and the space $\mathcal{H}_0$, it suffices to show that $f:H^1(\Omega)\rightarrow L^2(\Omega)$ is locally Lipschitz continuous and that $g:H^1(\Gamma)\rightarrow L^2(\Gamma)$ is locally Lipschitz continuous, both of which follow from rather standard arguments. 
Indeed, let $R>0$ and $u,v\in H^1(\Omega)$ be such that $\|u\|_{H^1(\Omega)}\le R$ and $\|v\|_{H^1(\Omega)}\le R.$
Using assumption (\ref{assm-3}), there holds
\begin{align}
\|f(u)-f(v)\|_{L^2(\Omega)} & \le C\left\||u-v|(1+|u|^2+|v|^2)\right\|_{L^2(\Omega)}  \notag \\ 
& \le C\|u-v\|_{L^6(\Omega)} \left( 1+\|u\|^2_{L^6(\Omega)}+\|v\|^2_{L^6(\Omega)} \right)  \notag \\ 
& \le Q(R)\|u-v\|_{H^1(\Omega)}.  \notag
\end{align}
For the component with $g$, let $u,v\in H^1(\Gamma)$ be such that $\|u\|_{H^1(\Gamma)}\le R$ and $\|v\|_{H^1(\Gamma)}\le R$.
Recall $H^1(\Gamma)\hookrightarrow L^p(\Gamma)$ for $p\in[1,\infty)$ because $\Gamma$ is a two-dimensional manifold.
Now using (\ref{assm-4}), we find 
\begin{align}
\|g(u)-g(v)\|_{L^2(\Gamma)} & \le C\||u-v|(1+|u|^{\rho-1}+|v|^{\rho-1})\|_{L^2(\Gamma)}  \notag \\ 
& \le C\|u-v\|_{L^q(\Gamma)} \left( 1+\|u\|^{\rho-1}_{L^{\frac{2q(\rho-1)}{q-2}}(\Gamma)}+\|v\|^{\rho-1}_{L^{\frac{2q(\rho-1)}{q-2}}(\Gamma)} \right) \quad \text{(for any $q>2$)}  \notag \\ 
& \le Q(R)\|u-v\|_{H^1(\Gamma)}.  \notag
\end{align}
This finishes the proof. 
\end{proof}

\section{The semiflow and attractors}  \label{s:sf}

\subsection{Semilinear equations with Gevrey semigroups}

Before we move on to the discussion of Problem {\textbf{P}}$_\alpha$, $\alpha\in[0,1]$, we address an abstract problem.
The following proposition will become useful when we seek the existence of local mild solutions to abstract ODE,
\begin{equation}   \label{abstract-ode}
\left\{
\begin{array}{ll} \displaystyle\frac{dU}{dt} = AU + F(U) & \text{for}\ t>0, \\ 
U(0)=U_0. \end{array}
\right.
\end{equation}
A mild solution is a function $U$ in the variation of parameters form,
\begin{equation}  \label{mild-2}
U(t)=e^{At}U_0 + \int_0^t e^{A(t-s)}F(U(s)) ds.
\end{equation}

\begin{theorem}  \label{t:sg-theory}
Assume the (unbounded) operator $A$ with domain $D(A)$ is the infinitesimal generator of a $C_0$-semigroup of contractions $e^{At}$ on a reflexive Banach space $X$, and let $F:D(F) \subset X \to X$ be a nonlinear operator.
Let $U_0 \in X$ be given.
\begin{enumerate}
\item Suppose $e^{At}$ satisfies the estimate, for all $t>0,$
\begin{equation} \label{eq:gevrey estimate}
\|Ae^{At}\|_{\mathcal{L}(X)} \leq Gt^{-\gamma}.
\end{equation}
Note well that this implies that $e^{At}$ is of Gevrey class $\delta$ for all $\delta > \gamma$ (cf. \cite[Chapter: {\em{Gevrey semigroups}}]{Taylor89}).
Assume moreover that $F:D(A^\theta) \to X$ is locally Lipschitz continuous for some $0 \leq \theta < \frac{1}{\gamma}$.
Then there exists a maximal time $T \in (0,\infty]$ such that \eqref{mild-2} has a unique solution in
\begin{equation*}
C([0,T);X) \cap C((0,T);D(A^{1/\gamma})). 
\end{equation*} 
Moreover, if $T < \infty$, then $\lim_{t \to T^-} \|U(t)\|_X = +\infty$.
\item In particular, if $e^{At}$ is analytic, then the previous statement (1) holds for $\gamma = 1$.
\end{enumerate}
\end{theorem}

\begin{remark}
Part (2) is well-known and applies to Problem {\textbf{P}}$_\alpha$ for $\alpha > 0$.
Concerning Problem {\textbf{P}}$_0$, we will apply part (1) of this theorem with $\gamma = 2$.
\end{remark}

Theorem \ref{t:sg-theory} is a corollary of the following two propositions.
\begin{proposition}  \label{t:sg-theory1}
Assume the (unbounded) operator $A$ with domain $D(A)$ is the infinitesimal generator of a $C_0$-semigroup of contractions $e^{At}$ on a reflexive Banach space $X$, and let $F:D(F) \subset X \to X$ be a nonlinear operator.
Let $U_0 \in X$ be given.

Suppose that for some $\gamma \geq 1$, $e^{At}$ satisfies the estimate \eqref{eq:gevrey estimate} for all $t>0.$
Assume moreover that $F:D(A^\theta) \to X$ is locally Lipschitz continuous for some $0 \leq \theta < \frac{1}{\gamma}$.
Then there exists a maximal time $T \in (0,\infty]$ such that \eqref{mild-2} has a unique solution in
$$
C([0,T);X) \cap C((0,T);D(A^\theta)).
$$
Moreover, if $T < \infty$, then $\lim_{t \to T^-} \|U(t)\|_X = +\infty$.
\end{proposition}

\begin{proof}
We will begin by showing that the equation,
\begin{equation} \label{change of var}
W(t) = t^{\gamma\theta}A^\theta e^{At}U_0 + t^{\gamma\theta}\int_0^t A^\theta e^{A(t-s)}F(s^{-\gamma\theta}A^{-\theta}W(s))ds,
\end{equation}
has a unique solution in $C([0,T];X)$ for some $T > 0$.
For $R > 0,T > 0$ denote by $X_{R,T}$ the complete metric space given by,
$$
X_{R,T} := C([0,T];B_X(R)) \quad \text{with} \quad \|V\|_{X_{R,T}} = \sup_{t \in [0,T]} \|V(t)\|_X,
$$
where $B_X(R) := \{x \in X: \|x\|_X \leq R\}$.
We want to show that the map given by
$$
\psi(W)(t) := t^{\gamma\theta}A^\theta e^{At}U_0 + t^{\gamma\theta}\int_0^t A^\theta e^{A(t-s)}F(s^{-\theta}A^{-\gamma\theta}W(s))ds
$$
is a contraction on $X_{R,T}$ for appropriately chosen $R>0$ and $T>0$.

First we recall that by interpolation (e.g. \cite[Theorem II.5.34]{Engel&Nagel00}) there exists $L > 0$ such that, for all $x\in D(A)$, the moment inequality holds
$$
\|A^\theta x\|_X \leq L\|x\|^{1-\theta}_X\|Ax\|^\theta_X.
$$
Then using the assumption on $e^{At}$, this gives, for all $x\in D(A)$,
$$
\|A^\theta e^{At}x\|_X \leq LG^\theta t^{-\gamma \theta}\|x\|_X,
$$
and since $D(A)$ is dense in $X$, it follows that, for all $t > 0$,
\begin{equation}  \label{sg bound}
\|A^\theta e^{At}\|_{\mathcal{L}(X)} \leq LG^\theta t^{-\gamma \theta}.
\end{equation}
Next, since $F:D(A^\theta) \to X$ is locally Lipschitz, there exists $K = K(R)$ such that, for all $x,y\in B_X(R),$
\begin{equation*}
\|F(A^{-\theta}x)-F(A^{-\theta}y)\|_X \leq K\|x-y\|_X.
\end{equation*}
We will also use below the following kernel estimates:
\begin{align}
\int_0^t (t-s)^{-\gamma \theta}ds & = \int_0^t s^{-\gamma \theta}ds  \notag \\ 
& = \frac{t^{1-\gamma \theta}}{1-\gamma \theta},  \label{kernel1}
\end{align}
and
\begin{align}
\int_0^t (t-s)^{-\gamma \theta}s^{-\gamma \theta} ds & = 2\int_0^{t/2} (t-s)^{-\gamma \theta}s^{-\gamma \theta}ds   \notag \\ 
& \leq  2\left(\frac{t}{2}\right)^{-\gamma \theta} \int_0^{t/2} s^{-\gamma \theta}ds  \notag \\
& = \frac{4^{\gamma \theta}t^{1-2\gamma \theta}}{1-\gamma \theta}.\label{kernel2}
\end{align}

Now observe that for $W \in X_{R,T}$, $t > 0$,
\begin{align}
\|\phi(W)(t)\|_X & \leq t^{\gamma\theta}\|A^\theta e^{At}\|_{\mathcal{L}(X)}\|U_0\|_X + t^{\gamma\theta}\int_0^t \|A^\theta e^{A(t-s)}\|_{\mathcal{L}(X)}\|F(s^{-\gamma\theta}A^{-\theta}W(s))-F(0)\|_X ds  \notag \\ 
& + t^{\gamma\theta}\int_0^t \|A^\theta e^{A(t-s)}\|_{\mathcal{L}(X)}\|F(0)\|_X ds  \notag \\ 
& \leq LG^\theta\|U_0\|_X + LG^\theta K t^{\gamma\theta}\int_0^t (t-s)^{-\gamma \theta}s^{-\gamma\theta}\|W(s)\|_X ds  \notag \\ 
& + LG^\theta t^{\gamma\theta}\int_0^t (t-s)^{-\gamma \theta} \|F(0)\|_X ds  \notag \\ 
& \leq LG^\theta\|U_0\|_X + LG^\theta KR \frac{4^{\gamma \theta}t^{1-\gamma \theta}}{1-\gamma \theta} + LG^\theta\|F(0)\|_X  \frac{t}{1-\gamma \theta}.  \notag
\end{align}
Therefore,
$$
\sup_{t \in [0,T]} \|\phi(W)(t)\|_X \leq LG^\theta\left(\|U_0\|_X +  \frac{4^{\gamma \theta}KRT^{1-\gamma \theta}+ T\|F(0)\|_X}{1-\gamma \theta}\right).
$$
Similarly, for $W,V \in X_{R,T}$, $t > 0$, we have
\begin{multline}
\|\phi(W)(t)-\phi(V)(t)\|_X \leq t^{\gamma\theta}\int_0^t \|A^\theta e^{A(t-s)}\|_{\mathcal{L}(X)}\|F(s^{-\gamma\theta}A^{-\theta}W(s))-F(s^{-\gamma\theta}A^{-\theta}V(s))\|_X ds 
\\
\leq LG^\theta K t^{\gamma\theta} \int_0^t (t-s)^{-\gamma \theta}s^{-\gamma\theta} \|W(s)-V(s)\|_X ds
\leq LG^\theta K \frac{4^{\gamma \theta}t^{1-\gamma \theta}}{1-\gamma \theta} \|W-V\|_{X_{R,T}},  \notag
\end{multline}
so that
$$
\|\phi(W)-\phi(V)\|_{X_{R,T}} \leq LG^\theta K \frac{4^{\gamma \theta}T^{1-\gamma \theta}}{1-\gamma \theta} \|W-V\|_{X_{R,T}}.
$$
Fix $R = 2LG^\theta \|U_0\|_X$.
Then it follows to choose $T > 0$ small enough such that
$$
LG^\theta\left(\|U_0\|_X +  \frac{4^{\gamma \theta}KRT^{1-\gamma \theta}+ T\|F(0)\|_X}{1-\gamma \theta}\right) \leq R,
$$
and
$$
LG^\theta K \frac{4^{\gamma \theta}T^{1-\gamma \theta}}{1-\gamma \theta} < 1.
$$
Then it follows that $\phi$ is a well-defined contraction on $X_{R,T}$.
It therefore has a unique fixed point, which is the same as saying \eqref{change of var} has a unique solution in $C([0,T];X)$.

To finish, let $U(t) = t^{-\gamma\theta}A^{-\theta} W(t)$ for $t > 0$, $U(0) = U_0$.
Then $U \in C((0,T);D(A^\theta))$, and for $t > 0$ we see by multiplying \eqref{change of var} by $t^{-\gamma\theta}A^{-\theta}$ that $U(t)$ satisfies \eqref{mild-2}.
It remains to show that $\lim_{t \to 0+} \|U(t) - U_0\|_X = 0$.
Observe:
\begin{align}
\|U(t) - U_0\|_X & \leq \|e^{At}U_0-U_0\|_X + \int_0^t \|F(s^{-\gamma\theta}A^{-\theta}W(s))-F(0)\|_Xds + t\|F(0)\|_X  \notag \\
& \leq \|e^{At}U_0-U_0\|_X + K\|W\|_{C([0,T];X)}\frac{t^{1-\gamma \theta}}{1-\gamma \theta} + t\|F(0)\|_X,  \notag
\end{align}
which converges to zero as $t \to 0$ by the strong continuity of $e^{At}$.
Hence $U \in C([0,T];X)$.

Finally, observe that if $\lim_{t \to T^-} \|U(T)\|_X = R < +\infty$, then by the argument just given, we can choose a $T_1 > 0$, depending on $R$ such that, for all $\epsilon > 0$, there exists $V \in C([0,T_1];X) \cap C((0,T_1];D(A^\theta))$ satisfying \eqref{mild-2} with $U_0$ replaced by $U(T-\epsilon)$.
Since $T_1$ depends on $R$ and not on $\epsilon$, we can choose $\epsilon = \frac{T_1}{2}$.
Then setting $\widetilde{U}(t) = U(t)$ for $t \in [0,T]$ and $\widetilde{U}(t) = V(t-T+\frac{T_1}{2})$ for $t \in [T,T+\frac{T_1}{2}]$, we see that $\widetilde{U} \in C([0,T+\frac{T_1}{2}];X) \cap C((0,T+\frac{T_1}{2}];D(A^\theta))$ solves \eqref{mild-2}, so $T$ is not maximal.
\end{proof}

\begin{remark} \label{smoothing bound 1}
In the above proof, $R$ was chosen proportional to $\|U_0\|_X$ and $T$ was chosen as a function of $R$.
We derive from the proof that, for small enough $t > 0$, we have $\|A^\theta U(t)\|_X \leq Ct^{-\gamma \theta}\|U_0\|_X$ where $C$ is a constant (for instance $C = 2LG^\theta$ as in the proof).
\end{remark}

The next proposition requires a lemma.
\begin{lemma} \label{holder f}
Assume the (unbounded) operator $A$ with domain $D(A)$ is the infinitesimal generator of a $C_0$-semigroup of contractions $e^{At}$ on a reflexive Banach space $X$, such that for some $\gamma \geq 1$, $e^{At}$ satisfies the estimate \eqref{eq:gevrey estimate} for all $t>0.$
Suppose $f \in L^1(0,T;X)$ is locally H\"older continuous on $(0,T)$.
Then for any $U_0 \in X$, there exists a unique $U \in C([t_0,T];X)$ which is a mild solution of
\begin{equation}   \label{abstract-ode-f}
\left\{
\begin{array}{ll} \displaystyle\frac{dU}{dt} = AU + f & \text{for}\ t>0, \\ 
U(t_0)=U_0, \end{array}
\right.
\end{equation}
and we have $U \in C((t_0,T);D(A^{1/\gamma}))$.
\end{lemma}

\begin{proof}
By definition of mild solution, we define, for all $t\in[t_0,T]$,
\begin{equation*}
U(t) = e^{A(t-t_0)}U_0 + \int_{t_0}^t e^{A(t-s)}f(s)ds.
\end{equation*}
Since $e^{At}$ satisfies \eqref{eq:gevrey estimate}, it follows that $t \mapsto A^{1/\gamma}e^{A(t-t_0)}U_0$ is continuous for $t_0 < t < T$.
We need to show that $\int_{t_0}^t e^{A(t-s)}f(s)ds \in D(A^{1/\gamma})$ for all $t_0 < t < T$ and that it is continuous in $t$.
Note that $\int_{t_0}^t e^{A(t-s)}f(t)ds$ is in $D(A)$ with $A\int_{t_0}^t e^{A(t-s)}f(t)ds = (e^{At}-e^{At_0})f(t)$, which is also continuous in $t$ for $t_0 < t < T$.
So it suffices to show that
$$
V(t) := \int_{t_0}^t e^{A(t-s)}(f(s)-f(t))ds
$$
is in $D(A^{1/\gamma})$ and is continuous in $t$.

We essentially follow the proof in \cite[Theorem 4.3.2]{Pazy83}.
Indeed, the approximations $V_\epsilon$ defined by
$$
V_\epsilon(t) := \int_{t_0}^{t-\epsilon} e^{A(t-s)}(f(s)-f(t))ds, \quad t \geq \epsilon
$$
and $V(t) = 0$ for $t < \epsilon$ converge to $V(t)$ as $\epsilon \to 0$ and satisfy $V(t) \in D(A)$ by \eqref{eq:gevrey estimate}.
We want to show that $A^{1/\gamma}V_\epsilon(t)$ converges.
For a fixed  $\delta_t > 0$ small enough, we can write
\begin{equation*}
\|f(s)-f(t)\| \leq C|s-t|^\beta, \quad \forall |s-t| < \delta_t,
\end{equation*}
for some $\beta \in (0,1)$, hence
\begin{equation} \label{f estimate}
\|A^{1/\gamma}e^{A(t-s)}(f(s)-f(t))\| \leq C|s-t|^{\beta-1} + LG^{1/\gamma}\delta_t^{-1}\|f(s)-f(t)\| \quad \forall s \in [t_0,t].
\end{equation}
Since the right-hand side is integrable, we conclude that
\begin{equation*}
\lim_{\epsilon \to 0} A^{1/\gamma}V_\epsilon(t) = \int_{t_0}^t A^{1/\gamma}e^{A(t-s)}(f(s)-f(t))ds,
\end{equation*}
by the dominated convergence theorem.
Since $A$ is closed we get $V(t) \in D(A^{1/\gamma})$ with $A^{1/\gamma}V(t)$ equal to the right-hand side.

Finally, to see that $A^{1/\gamma}V(t)$ is continuous, we can split it as
\begin{equation} \label{V split}
A^{1/\gamma}V(t) = \int_{t_0}^{t_0 + \delta} A^{1/\gamma}e^{A(t-s)}(f(s)-f(t))ds + \int_{t_0 + \delta}^t A^{1/\gamma}e^{A(t-s)}(f(s)-f(t))ds
\end{equation}
for arbitrary $\delta > 0$.
The second term on the right-hand side is continuous in $t$ by the observations made above.
The first term, on the other hand, can be made arbitrarily small by choosing $\delta > 0$ small.
This completes the proof.
\end{proof}

\begin{remark} \label{smoothing bound-2}
We note for future reference that, using \eqref{V split} and \eqref{f estimate}, we can estimate
\begin{equation} \label{AU estimate}
\|A^{1/\gamma}U(t)\| \leq LG^{1/\gamma}(t-t_0)^{-1}\|U_0\| + LG^{1/\gamma}\delta_t^{-1}\int_{t_0}^{t-\delta_t}\|f(s)-f(t)\|ds + 2\|A^{1/\gamma-1}f(t)\|.
\end{equation}
\end{remark}

\begin{proposition} \label{t:sg-theory2}
Assume the (unbounded) operator $A$ with domain $D(A)$ is the infinitesimal generator of a $C_0$-semigroup of contractions $e^{At}$ on a reflexive Banach space $X$, and let $F:D(F) \subset X \to X$ be a nonlinear operator.
Suppose that for some $\gamma \geq 1$, $e^{At}$ satisfies the estimate \eqref{eq:gevrey estimate} for all $t>0.$
Assume moreover that $F:D(A^\theta) \to X$ is locally Lipschitz continuous for some $0 \leq \theta < \frac{1}{\gamma}$.

Then for any initial data $U_0 \in D(A^\theta)$ and $t_0 \geq 0$ there exists a maximal time $T \in (t_0,\infty]$ such that
\begin{equation}   \label{abstract-odet_0}
\left\{
\begin{array}{ll} \displaystyle\frac{dU}{dt} = AU + F(U) & \text{for}\ t>0, \\ 
U(t_0)=U_0, \end{array}
\right.
\end{equation}
has a unique mild solution in
$$
C([t_0,T);X) \cap C((t_0,T);D(A^{1/\gamma})).
$$
Moreover, if $T < \infty$, then $\lim_{t \to T^-} \|U(t)\|_X = +\infty$.
\end{proposition}

\begin{proof}
Using the same fixed point argument as in the proof of Proposition \ref{t:sg-theory1}, and keeping in mind that $U_0 \in D(A^\theta)$, we deduce that there exists a unique $W \in C([t_0,T];X)$, for small enough $T > 0$, such that
\begin{equation} \label{eq:var par1}
W(t) = e^{A(t-t_0)}A^\theta U_0 + \int_{t_0}^t A^\theta e^{A(t-s)}F(A^{-\theta}W(s))ds.
\end{equation}
We will prove, further, that $W(\cdot)$ is locally H\"older continuous in time on $(t_0,T)$.
The proof follows the same lines as in \cite[Theorem 6.3.1]{Pazy83}.
For $t_0 \leq t \leq t+h \leq T$ write
\begin{multline}
W(t+h) - W(t) = e^{A(t+h-t_0)}A^\theta U_0 - e^{A(t-t_0)}A^\theta U_0 
+ \int_{t_0}^t (e^{Ah}-I)A^\theta e^{A(t-s)}F(A^{-\theta}W(s))ds \\
+ \int_t^{t+h} A^\theta e^{A(t+h-s)}F(A^{-\theta}W(s))ds = I_1 + I_2 + I_3.  \notag
\end{multline}
In order to estimate these terms, first note that $\|F(A^{-\theta}W(t))\|_X \leq C$ for all $t \in [t_0,T]$, where $C$ depends on $\sup_{t \in [t_0,T]} \|W(t)\|_X$ and the Lipschitz constant of $F:D(A^\theta) \to X$.
Second, using \eqref{sg bound} we deduce that for any $0 < \beta < 1$, for any $Y \in D(A^\beta)$,
\begin{multline}
\|e^{At}Y - Y\|_X = \left\|\int_0^t Ae^{As}Y ds \right\|_X = \left\|\int_0^t A^{1-\beta}e^{As}A^\beta Y ds \right\|_X \\
\leq C\int_0^t s^{-(1-\beta)\gamma}\|A^\beta Y\|_X ds
= C_\beta t^{1-\gamma(1-\beta)}\|A^\beta Y\|_X .  \notag
\end{multline}
Combining these estimates we deduce
\begin{align*}
\|I_1\|_X &\leq C_\beta h^{1-\gamma(1-\beta)}\|A^{\theta + \beta}e^{A(t-t_0)}U_0\|_X \leq C_\beta h^{1-\gamma(1-\beta)} (t-t_0)^{-\gamma(\theta + \beta)},\\
\|I_2\|_X &\leq C_\beta h^{1-\gamma(1-\beta)} \int_{t_0}^t (t-s)^{-\gamma \theta}ds = C_{\beta,\theta} h^{1-\gamma(1-\beta)}(t-t_0)^{1-\gamma \theta},\\
\|I_3\|_X &\leq C_\theta h^{1-\gamma\theta} \leq C_{\beta,\theta}h^{1-\gamma(1-\beta)},
\end{align*}
where $\beta$ is chosen in the range $1 - \frac{1}{\gamma} < \beta < 1 - \theta$.
Then these estimates prove that $W(\cdot)$ is locally H\"older continuous on $(t_0,T)$ with H\"older exponent $1-\gamma(1-\beta)$.
Since $F:D(A^\theta) \to X$ is locally Lipschitz, it follows then that $t \mapsto F(A^{-\theta}Y(t))$ is also locally H\"older continuous in time on $(t_0,T)$.

Now we can appeal to Lemma \ref{holder f}, which tells us that
\begin{equation*}
\left\{
\begin{array}{ll} \displaystyle\frac{dU}{dt} = AU + F(A^{-\theta} W) & \text{for}\ t>0, \\ 
U(t_0)=U_0 \in D(A^\theta), \end{array}
\right.
\end{equation*}
has a unique mild solution $U \in C([t_0,T];D(A^\theta)) \cap C((t_0,T];D(A^{1/\gamma}))$ given by
\begin{equation} \label{eq:var par 2}
U(t) = e^{A(t-t_0)}U_0 + \int_{t_0}^t e^{A(t-s)}F(A^{-\theta}W(s))ds.
\end{equation}
The right-hand side is contained in $D(A)$, hence also in $D(A^\theta)$, for $t > t_0$.
We deduce that 
\begin{equation*}
A^\theta U(t) = e^{A(t-t_0)} A^\theta U_0 + \int_{t_0}^t A^\theta e^{A(t-s)}F(A^{-\theta}W(s))ds,
\end{equation*}
and so by uniqueness $A^\theta U(t) = W(t)$.
It follows that $U$ is a mild solution of \eqref{abstract-odet_0}.
Uniqueness follows from the fact that the solution $W$ of \eqref{eq:var par1} is unique and from formula \eqref{eq:var par 2}.
\end{proof}

\begin{proof}[Proof of Theorem \ref{t:sg-theory}]
For $U_0 \in X$, \eqref{abstract-ode} has a unique solution
$U \in C([0,T];X) \cap C((0,T);D(A^\theta))$ by Proposition \ref{t:sg-theory1}.
By Proposition \ref{t:sg-theory2} with initial data $t_0 > 0$ and $U(t_0) \in D(A^\theta)$, we deduce that $U \in C((t_0,T];D(A^{1/\gamma}))$ for arbitrary $t_0 > 0$.
The conclusion follows.
\end{proof}

\begin{remark}
\label{smoothing bound final}
Examining the proof of Proposition \ref{t:sg-theory2} and Theorem \ref{t:sg-theory} and using Remarks \ref{smoothing bound 1} and \ref{smoothing bound-2}, we conclude that for $t > 0$ small enough, there is some constant $C$ such that
$\|A^{1/\gamma}U(t)\|_X \leq Ct^{-1}\|U_0\|_X$.
\end{remark}

\subsection{Well-posedness of Problem {\text{P}}$_\alpha$}

The definition of mild solution is from \cite{Ball77}.

\begin{definition}  \label{d:mild}
Let $T>0$, $\alpha\in[0,1]$. 
A function $\zeta\in C([0,T];\mathcal{H}_0)$ is called a mild solution of (\ref{abstract-p}) (and Problem {\textbf{P}}$_\alpha$, $\alpha\in[0,1]$) if and only if $\mathcal{F}(\zeta(\cdot))\in L^1(0,T;\mathcal{H}_0)$ and $\zeta$ satisfies the variation of constants formula for all $t\in[0,T],$
\begin{equation}  \label{mild}
\zeta(t)=e^{A_{\alpha}t}\zeta_0 + \int_0^t e^{A_{\alpha}(t-s)}\mathcal{F}(\zeta(s)) ds.
\end{equation}
The mild solution is called a global mild solution if it is a mild solution on $[0,T]$ for all $T>0.$
\end{definition}

\begin{theorem}  \label{t:strong-solns}
Assume $\alpha\in [0,1]$.
Let $T>0$ and $\zeta_0=(u_0,u_1,\gamma_0,\gamma_1)\in\mathcal{H}_0$. 
There exists a unique mild solution to Problem {\textbf{P}}$_\alpha$ given by (\ref{mild}) satisfying the additional regularity,
\begin{equation}  \label{regularity}
\begin{array}{ll}
\zeta\in C([0,T];\mathcal{H}_0) \cap C((0,T];D(A_\alpha)) \cap C^1((0,T];\mathcal{H}_0) & \text{if}~~\alpha > 0,  \\
\zeta\in C([0,T];\mathcal{H}_0) \cap C((0,T];D(A_0^{1/2})) & \text{if}~~\alpha = 0. 
\end{array}
\end{equation}
For any mild solution to Problem {\textbf{P}}$_\alpha$, $\alpha\in[0,1]$, on $[0,T]$, the map 
\begin{equation}  \label{C1-map}
t \mapsto \mathcal{E}(t) := \|\zeta(t)\|^2_{\mathcal{H}_0} + 2\int_\Omega F(u(t)) dx + 2\int_\Gamma G(u(t)) d\sigma  \quad \text{is} \ C^1([0,T])
\end{equation}
and the energy equation,
\begin{align}
\frac{d}{dt} \mathcal{E} + 2\omega\|\nabla u_t\|^2_{L^2(\Omega)} + 2\|u_t\|^2_{L^2(\Omega)} + 2\alpha\omega\|\nabla_\Gamma u_t\|^2_{L^2(\Gamma)} + 2\|u_t\|^2_{L^2(\Gamma)} = 0 \label{energy-p}
\end{align} 
holds for almost all $t\in[0,T]$. 
Furthermore, there holds, for any $T>0$ and for all $t\in[0,T]$,
\begin{equation}  \label{wkest-1}
\|\zeta(t)\|_{\mathcal{H}_0} \le Q(\|\zeta_0\|_{\mathcal{H}_0}).
\end{equation}
Therefore, $T=+\infty$ and any mild solution is globally bounded in $\mathcal{H}_0.$
\end{theorem}

\begin{proof}
First we show the local existence of a unique local mild solution to Problem {\textbf{P}}$_\alpha$, $\alpha\in[0,1].$
By Lemma \ref{t:operator-A}, the operator $A_{\alpha}$ with domain $D(A_{\alpha})\subset\mathcal{H}_0$ satisfies \eqref{eq:gevrey estimate} with either $\gamma = 1$ (if $\alpha > 0$) or $\gamma = 2$ (if $\alpha = 0$).
By Lemma \ref{t:operator-F}, the functional $\mathcal{F}:D(A_{\alpha}^\theta)\subset\mathcal{H}_0 \rightarrow \mathcal{H}_0$ is locally Lipschitz continuous for any $0 \leq \theta < 1$.
It follows from Theorem \ref{t:sg-theory} that for each $\zeta_0=(u_0,u_1,\gamma_0,\gamma_1)\in\mathcal{H}_0$, there exists $T>0$ and a unique mild solution given by (\ref{mild}) to Problem {\textbf{P}}$_\alpha$ such that \eqref{regularity} holds.

To show that the assertion \eqref{C1-map} holds, we appeal to the continuity properties of mild solutions (see \eqref{regularity}) and the sequential weak continuity of the functional $\mathcal{F}$ on $\mathcal{H}_0$ (this fact follows with simple modifications to the proof of \cite[Lemma 3.3]{Ball04} for example). 
The energy identity \eqref{energy-p} can be derived as in the proof of \cite[Theorem 3.6]{Ball04}.
(Of course, formally, the energy identity may be obtained by multiplying (\ref{pde}) by $u_t$ in $L^2(\Omega)$.)

To show (\ref{wkest-1}), integrate (\ref{energy-p}) over $(0,t)$ and apply (\ref{cons-4}), (\ref{cons-7}), (\ref{cons-5}), and (\ref{cons-8}) to obtain 
\begin{align}
\|\zeta(t)\|^2_{\mathcal{H}_0} + 2\int_0^t \|u_t(\tau)\|^2_{H^1(\Omega)} d\tau + 2\int_0^t\|u_t(\tau)\|^2_{H^1(\Gamma)} d\tau \le Q(\|\zeta_0\|_{\mathcal{H}_0}),   \label{wkest-2}
\end{align}
from which (\ref{wkest-1}) clearly follows.
This finishes the proof.
\end{proof}

The following proposition can be used to show that (mild) solutions to Problem {\textbf{P}}$_{\alpha}$, $\alpha\in[0,1]$, depend continuously on initial data. 

\begin{proposition}  \label{t:wk-unique}
Let $T>0$, $R>0$ and $\zeta_{01},\zeta_{02}\in\mathcal{H}_0$ be such that $\|\zeta_{01}\|_{\mathcal{H}_0}\le R$ and $\|\zeta_{02}\|_{\mathcal{H}_0}\le R$.
Any two mild solutions, $\zeta^1(t)$ and $\zeta^2(t)$, to Problem {\textbf{P}}$_\alpha$, $\alpha\in[0,1]$, on $[0,T]$ corresponding to the initial data $\zeta_{01}$ and $\zeta_{02}$, respectively, satisfy, for all $t\in[0,T]$,
\begin{equation}  \label{contdep}
\|\zeta^1(t)-\zeta^2(t)\|_{\mathcal{H}_0} \le e^{Q(R)t} \|\zeta_{01} - \zeta_{02}\|_{\mathcal{H}_0}.
\end{equation}
\end{proposition}

\begin{proof}
To prove the continuous dependence in initial conditions, let $\zeta_{01} = (u_{01},u_{11},\gamma_{01},\gamma_{11}), \zeta_{02} = (u_{02},u_{12},\gamma_{02},\gamma_{12})\in\mathcal{H}_0$ be such that $\|\zeta_{01}\|_{\mathcal{H}_0}\le R$ and $\|\zeta_{02}\|_{\mathcal{H}_0}\le R$ for some $R>0$. 
Let $\zeta^1(t)=(u^1(t),u_t^1(t),u_{\mid\Gamma}^1(t),u_{t\mid\Gamma}^1(t))$ and, respectively, $\zeta^2(t)=(u^2(t),u_t^2(t),u_{\mid\Gamma}^2(t),u_{t\mid\Gamma}^2(t))$ denote the corresponding solutions of Problem {\textbf{P}}$_\alpha$, $\alpha\in[0,1]$, on $[0,T]$ with the initial data $\zeta_{01}$ and $\zeta_{02}$. 
For all $t\in[0,T]$, set
\begin{align}
\bar\zeta(t) & := \zeta^1(t) - \zeta^2(t)  \notag \\
& = \left( u^1(t),u_t^1(t),u_{\mid\Gamma}^1(t),u_{t\mid\Gamma}^1(t) \right) - \left( u^2(t),u^2_t(t),u^2_{\mid\Gamma}(t),u^2_{t\mid\Gamma}(t) \right)  \notag \\
& =: \left( \bar u(t),\bar u_t(t),\bar u_{\mid\Gamma}(t),\bar u_{t\mid\Gamma}(t) \right),  \notag
\end{align}
and 
\begin{align}
\bar\zeta_0 & := \zeta_{01} - \zeta_{02}  \notag \\ 
& = \left( u_{01},u_{11},\gamma_{01},\gamma_{11} \right) - \left( u_{02},u_{12},\gamma_{02},\gamma_{12} \right) \notag \\
& = \left( u_{01}-u_{02},u_{11}-u_{12},\gamma_{01}-\gamma_{02},\gamma_{11}-\gamma_{12} \right)  \notag \\
& =: (\bar u_0,\bar u_1,\bar \gamma_0,\bar \gamma_1).  \notag
\end{align}
Then $\bar u$ satisfies the IBVP

\begin{equation}  \label{cde-1}
\left\{ \begin{array}{ll} \bar u_{tt}-\omega\Delta \bar u_t+\bar u_t-\Delta \bar u+\bar u+f(u^1)-f(u^2) = 0 & \text{in}\quad (0,T)\times\Omega \\ 
\bar u_{tt}+\partial_{\mathbf{n}}(\bar u+\omega\bar u_t)-\alpha\omega\Delta_\Gamma \bar u_t+\bar u_t-\Delta_\Gamma \bar u + \bar u + g(u^1)-g(u^2) = 0 & \text{on} \quad (0,T)\times\Gamma \\ 
\bar u(0,\cdot)=\bar u_0,  \quad \bar u_t(0,\cdot)=\bar u_1 & \text{at} \quad \{0\}\times\Omega \\
\bar u_{\mid\Gamma}(0,\cdot)=\bar \gamma_0, \quad \bar u_{t\mid\Gamma}(0,\cdot)=\bar \gamma_1 & \text{at} \quad \{0\}\times\Gamma.
\end{array} \right.
\end{equation}
Multiply (\ref{cde-1})$_1$ by $2\bar u_t$ in $L^2(\Omega)$ to yield, for almost all $t\in[0,T]$,
\begin{align}
\frac{d}{dt} & \|\bar\zeta\|^2_{\mathcal{H}_0} + 2\omega\|\bar u_t\|^2_{H^1(\Omega)} + 2\alpha\omega\|\bar u_t\|^2_{H^1(\Gamma)}  \notag \\ 
& \le -2\langle f(u^1)-f(u^2),\bar u_t \rangle_{L^2(\Omega)} - 2\langle g(u^1)-g(u^2),\bar u_t \rangle_{L^2(\Gamma)}.  \label{cde-2} 
\end{align}
For terms on the right-hand side, first consider when $\alpha\in[0,1]$,
\begin{align}
-2\langle f(u^1)-f(u^2),\bar u_t \rangle_{L^2(\Omega)} & \le 2\|(f(u^1)-f(u^2))\bar u_t\|_{L^1(\Omega)}  \notag \\ 
& \le 2\|f(u^1)-f(u^2)\|_{L^{6/5}(\Omega)}\|\bar u_t\|_{L^{6}(\Omega)}  \notag \\ 
& \le 2\ell_1\|\bar u\|_{L^6(\Omega)}\left( 1+\|u^1\|^2_{L^{2\cdot\frac{3}{2}}(\Omega)}+\|u^2\|^2_{L^{2\cdot\frac{3}{2}}(\Omega)} \right)\|\bar u_t\|_{L^6(\Omega)}  \notag \\ 
& \le 2\ell_1 \left( 1+\|u^1\|^4_{H^1(\Omega)}+\|u^2\|^4_{H^1(\Omega)} \right) \left( \ep\|\bar u\|^2_{H^1(\Omega)}+\frac{1}{4\ep}\|\bar u_t\|^2_{H^1(\Omega)} \right)  \notag \\ 
& \le Q(R) \left( \ep\|\bar u\|^2_{H^1(\Omega)}+\frac{1}{4\ep}\|\bar u_t\|^2_{H^1(\Omega)} \right),  \label{cde-3}
\end{align}
where the last inequality follows from the global bound on the weak solutions (\ref{wkest-1}), and $\ep>0$ will be chosen later.
In a similar fashion, we find for $g$, 
\begin{align}
-2\langle g(u^1)-g(u^2),\bar u_t \rangle_{L^2(\Gamma)} & \le Q(R) \left( \ep\|\bar u\|^2_{H^1(\Gamma)}+\frac{1}{4\ep}\|\bar u_t\|^2_{H^1(\Gamma)} \right).  \label{cde-4}
\end{align}
We now choose $\ep>0$ so that $2\alpha\omega-Q(R)\frac{1}{4\ep}>0$; hence eliminating the terms with $\|\bar u_t\|^2_{H^1(\Omega)}$ and $\|\bar u_t\|^2_{H^1(\Gamma)}$ from (\ref{cde-2}).
Combining the remaining terms in (\ref{cde-2})--(\ref{cde-4}) leaves us with 
\begin{align}
\frac{d}{dt} \|\bar\zeta\|^2_{\mathcal{H}_0} & \le Q(R) \left( \|\bar u\|^2_{H^1(\Omega)} + \|\bar u\|^2_{H^1(\Gamma)} \right)  \notag \\ 
& \le Q(R) \|\bar\zeta\|^2_{\mathcal{H}_0}.  \label{cde-5} 
\end{align}
Integrating (\ref{cde-5}) over $[0,T]$, we arrive at (\ref{contdep}).
\end{proof}

We formalize the dynamical system associated with Problem {\textbf{P}}$_\alpha$.

\begin{corollary}  \label{sf-a}
Let $\zeta_0=(u_0,u_1,\gamma_0,\gamma_1)\in\mathcal{H}_0$ and let $u$ be the unique (mild) solutions of Problem {\textbf{P}}$_\alpha$, $\alpha\in[0,1]$. 
For each $\alpha\in[0,1]$, define the family of maps $S_\alpha=(S_\alpha(t))_{t\geq 0}$ by 
\begin{align}
& S_\alpha(t)\zeta_0(x):=  \notag \\
& (u(t,x,u_0,u_1,\gamma_0,\gamma_1), u_t(t,x,u_0,u_1,\gamma_0,\gamma_1), u_{\mid\Gamma}(t,x,u_0,u_1,\gamma_0,\gamma_1), u_{t\mid\Gamma}(t,x,u_0,u_1,\gamma_0,\gamma_1))  \notag
\end{align}
is the semiflow on $\mathcal{H}_0$ generated by Problem {\textbf{P}}$_\alpha$. 
The operators $S(t)$ satisfy
\begin{enumerate}
	\item $S_\alpha(t+s)=S_\alpha(t)S_\alpha(s)$ for all $t,s\geq 0$.
	\item $S_\alpha(0)=I_{\mathcal{H}_0}$ (the identity on $\mathcal{H}_0$).
	\item $S_\alpha(t)\zeta_0\rightarrow S_\alpha(t_0)\zeta_0$ for every $\zeta_0\in\mathcal{H}_0$ when $t\rightarrow t_0$.
\end{enumerate}

Additionally, each mapping $S_\alpha(t):\mathcal{H}_0\rightarrow\mathcal{H}_0$ is Lipschitz continuous, uniformly in $t$ on compact intervals; i.e., for each $T>0$ and for all $\zeta_{01},\zeta_{02}\in\mathcal{H}_0$ in which $\|\zeta_{01}\|_{\mathcal{H}_0}\le R$ and $\|\zeta_{02}\|_{\mathcal{H}_0}\le R$, for all $t\in[0,T]$,
\begin{equation}  \label{SLipcont}
\|S_\alpha(t)\zeta_{01} - S_\alpha(t)\zeta_{02}\|_{\mathcal{H}_0} \le e^{Q(R)t}\|\zeta_{01} - \zeta_{02}\|_{\mathcal{H}_0}.
\end{equation}
\end{corollary}

\begin{proof}
The semigroup properties (1) and (2) are well-known and apply to a general class of abstract Cauchy problems possessing many applications (see \cite{Babin&Vishik92,Morante79,Goldstein85,Tanabe79}; in particular, a proof of property (1) is given in \cite[\S1.2.4]{Milani&Koksch05}). 
The continuity in $t$ described by property (3) follows from the definition of the solution (this also establishes strong continuity of the operators when $t_0=0$). 
The continuity property (\ref{SLipcont}) follows from (\ref{contdep}).
\end{proof}

Due to the complicated nature of the domain of the operator $A_\alpha$, we will not report any results on strong solutions to Problem {\textbf{P}}$_\alpha$ (cf. e.g. \cite[Theorem 2.5.6]{Zheng04}).

\subsection{Dissipativity}  \label{s:dissipativity}

The main result depends on the following proposition. 
It can be found in \cite[Lemma 2.7] {Belleri&Pata01}.

\begin{proposition}  \label{absset}
Let $X$ be an arbitrary Banach space, and $Z\subset C([0,\infty);X)$. 
Suppose that there is a functional $E:X\rightarrow\mathbb{R}$ such that, for every $z\in Z$,
\[
\sup_{t\geq 0} E(z(t))\geq -r \quad \text{and} \quad E(z(0))\leq R
\]
for some $r,R\geq 0$. 
In addition, assume that the map $t\mapsto E(z(t))$ is $C^1([0,\infty))$ for every $z\in Z$ and that for almost all $t\geq 0$, the differential inequality holds
\[
\frac{d}{dt} E(z(t)) + m\|z(t)\|^2_X \le C,
\]
for some $m>0$, $C\ge 0$, both independent of $z\in Z$. Then, for every $\iota>0$, there exists $t_0\ge 0$, depending on $R$ and $\iota$, such that for every $z\in Z$ and for all $t\ge t_0$,
\[
E(z(t)) \le \sup_{\xi\in X}\{ E(\xi):m\|\xi\|^2_X\le C+\iota \}.
\]
Furthermore, $t_0=(r+R)/\iota$.
\end{proposition} 

\begin{lemma}  \label{t:abs-set}
There exists $R_0>0$ so that the set 
\begin{equation}  \label{abs-set}
\mathcal{B}_0:=\{\zeta\in\mathcal{H}_0:\|\zeta\|_{\mathcal{H}_0}\le R_0\}
\end{equation}
satisfies the following: for all $\alpha\in[0,1]$, $\omega\in(0,1]$, and for any bounded subset $B\subset\mathcal{H}_0$, there is $t_0=t_0(\|B\|_{\mathcal{H}_0},c_i)\ge0$, $i=1,2,3,4$, such that $S_\alpha(t)B\subset\mathcal{B}_0$ for all $t\ge t_0.$
\end{lemma}

\begin{proof}
Multiply equation (\ref{pde}) by $\ep u$ in $L^2(\Omega)$, where $\ep>0$ will be chosen later, and add the resulting differential identity to the ``energy'' identity (\ref{energy-p}) to obtain the following identity which holds for almost all $t\ge0,$
\begin{align}
\frac{d}{dt} & \left\{ \|\zeta\|^2_{\mathcal{H}_0} + \ep\langle u_t,u \rangle_{L^2(\Omega)} + \ep\langle u_t,u \rangle _{L^2(\Gamma)} + 2\int_\Omega F(u) dx + 2\int_\Gamma G(u) d\sigma \right\}  \notag \\ 
& + 2\omega\|\nabla u_t\|^2_{L^2(\Omega)} + 2\|u_t\|^2_{L^2(\Omega)} + \ep\omega\langle \nabla u_t,\nabla u \rangle_{L^2(\Omega)} + \ep\langle u_t,u \rangle_{L^2(\Omega)} + \ep\|u\|^2_{H^1(\Omega)}  \notag \\ 
& + 2\alpha\omega\|\nabla_\Gamma u_t\|^2_{L^2(\Gamma)} + 2\|u_t\|^2_{L^2(\Gamma)} + \ep\alpha\omega\langle \nabla_\Gamma u_t,\nabla_\Gamma u \rangle_{L^2(\Gamma)} + \ep\langle u_t,u \rangle_{L^2(\Gamma)} + \ep\|u\|^2_{H^1(\Gamma)}  \notag \\ 
& + \ep\langle f(u),u \rangle_{L^2(\Omega)} + \ep\langle g(u),u \rangle_{L^2(\Gamma)} = \ep\|u_t\|^2_{L^2(\Omega)} + \ep\|u_t\|^2_{L^2(\Gamma)}.  \label{stest-3}
\end{align}

Define the following functional for all $\zeta=(u,v,\gamma,\delta)\in\mathcal{H}_0$ and for each $\ep>0$,
\begin{align}
\Psi(\zeta) := & \|\zeta\|^2_{\mathcal{H}_0} + \ep\langle u,v \rangle_{L^2(\Omega)} + \ep\langle \gamma,\delta \rangle _{L^2(\Gamma)} + 2\int_\Omega F(u) dx + 2\int_\Gamma G(\gamma) d\sigma.  \label{stest-4}
\end{align}
(Also, on trajectories $\zeta(t)=(u(t),u_t(t),u_{\mid\Gamma}(t),u_{t\mid\Gamma}(t))$, $t>0$, we denote $\Psi(\zeta)$ by $\Psi(t).$)
Observe, with the Cauchy--Schwarz inequality, the Poincar\'{e} inequality (\ref{Poincare}), the trace embedding $H^1(\Omega)\hookrightarrow L^2(\Gamma)$, and (\ref{cons-4}), (\ref{cons-5}), (\ref{cons-7}) and (\ref{cons-8}), a straight forward calculation shows there are constants $C_1,C_2>0$, such that, for all $t\ge0,$
\begin{align}
C_1\|\zeta(t)\|^2_{\mathcal{H}_0} - 2c_2 - 2c_4 \le \Psi(t),  \label{stest-11}
\end{align}
and
\begin{align}
\Psi(t) \le C_2\left( \|\zeta(t)\|^2_{\mathcal{H}_0} + \|u(t)\|^6_{H^1(\Omega)} + \|u(t)\|_{H^1(\Omega)} + \|u(t)\|^{\rho+1}_{H^1(\Gamma)} + \|u(t)\|_{H^1(\Gamma)} \right).  \label{stest-11b}
\end{align}
Moreover, the map $t\mapsto \Psi(t)$ is $C^1([0,\infty)).$

Since $\omega\in(0,1]$, in (\ref{stest-3}) we estimate,
\begin{align}
\ep\omega\langle\nabla u_t,\nabla u\rangle_{L^2(\Omega)} & \ge -\frac{\omega}{4}\|\nabla u_t\|^2_{L^2(\Omega)} - \ep^2\|\nabla u\|^2_{L^2(\Omega)},  \label{stest-9}
\end{align}
\begin{align}
\ep\langle u_t, u\rangle_{L^2(\Omega)} & \ge -\frac{\omega}{2}\|u_t\|^2_{L^2(\Omega)} - \frac{2\ep^2}{\omega}\|u\|^2_{L^2(\Omega)},  \label{stest-9b}
\end{align}
\begin{align}
\ep\alpha\omega\langle \nabla_\Gamma u_t,\nabla_\Gamma u \rangle_{L^2(\Gamma)} & \ge -\frac{\alpha\omega}{4}\|\nabla_\Gamma u_t\|^2_{L^2(\Gamma)} - \ep^2\|\nabla_\Gamma u\|^2_{L^2(\Gamma)},  \label{stest-10}
\end{align}
\begin{align}
\ep\langle u_t, u\rangle_{L^2(\Gamma)} & \ge -\frac{\omega}{2}\|u_t\|^2_{L^2(\Gamma)} - \frac{2\ep^2}{\omega}\|u\|^2_{L^2(\Gamma)}.  \label{stest-10b}
\end{align}
We see that combining (\ref{stest-3})--(\ref{stest-10b}), with (\ref{cons-3}) and (\ref{cons-6}), there holds for almost all $t\ge0,$
\begin{align}
\frac{d}{dt} & \left\{ \|\zeta\|^2_{\mathcal{H}_0} + \ep\langle u_t,u \rangle_{L^2(\Omega)} + \ep\langle u_t,u \rangle _{L^2(\Gamma)} + 2\int_\Omega F(u) dx + 2\int_\Gamma G(u) d\sigma \right\}  \notag \\ 
& + \frac{7}{8}\omega \|\nabla u_t\|^2_{L^2(\Omega)} + \left( 2-\frac{\omega}{2}-\ep \right) \|u_t\|^2_{L^2(\Omega)} + \ep\left( 1-\ep \right)\|\nabla u\|^2_{L^2(\Omega)} + \ep\left(\mu_1-\frac{2\ep}{\omega}\right) \|u\|^2_{L^2(\Omega)}  \notag \\ 
& + \frac{7}{8}\alpha\omega \|\nabla_\Gamma u_t\|^2_{L^2(\Gamma)} + \left( 2-\frac{\omega}{2}-\ep \right) \|u_t\|^2_{L^2(\Gamma)} + \ep\left( 1-\ep \right)\|\nabla_\Gamma u\|^2_{L^2(\Gamma)} + \ep\left(\mu_2-\frac{2\ep}{\omega}\right) \|u\|^2_{L^2(\Gamma)}  \notag \\ 
& \le \ep(c_1 + c_3).  \label{stest-14}
\end{align}
Again, $\omega\in(0,1]$, so $2-\frac{\omega}{2}-\ep \ge \frac{3}{2}-\ep$.
Set
\[
\ep_1:=\min\left\{ \frac{\mu_1\omega}{2},\frac{\mu_2\omega}{2} \right\}.
\]
Then for all $\ep\in(0,\ep_1),$ the positivity of 
\[
m_1:=\min\left\{ \mu_1-\frac{2\ep}{\omega}, \mu_2-\frac{2\ep}{\omega} \right\}>0
\]
is assured. 
Fix any $\ep^*\in(0,\ep_1).$
After applying both Poincar\'{e} ineqialities \eqref{Poincare} and \eqref{bndry-Poincare}, we are able to choose a constant $\nu_1>0$ so that \eqref{stest-14} becomes,  
\begin{align}
\frac{d}{dt}\Psi + \nu_1 \|\zeta\|^2_{\mathcal{H}_0} \le C.  \label{stest-12}
\end{align}
Observe, $\nu_1\sim m_1 \lesssim 1$ for all $\alpha\in(0,1]$ and $\omega\in(0,1]$.

Let $\widetilde R>0$. 
For all $\zeta_0\in \mathcal{H}_0$ with $\|\zeta_0\|_{\mathcal{H}_0} \le \widetilde R$, the upper-bound in (\ref{stest-11b}) reads
\begin{equation}  \label{stest-zero}
\Psi(\zeta_0) \le C_2 \left( \widetilde R^2 + \widetilde{R}^6 + \widetilde{R}^{\rho+1} + 2\widetilde{R} \right).
\end{equation}
Hence, for all $\widetilde R>0$, there exists $R>0$ such that, for all $\zeta_0\in\mathcal{H}_0$ with $\|\zeta_0\|_{\mathcal{H}_0}\le \widetilde R$, then $\Psi(\zeta_0)\le R.$
From the lower-bound in (\ref{stest-11}), we immediately see that $\sup_{t\ge0}\Psi(\zeta(t)) \ge -2c_2-2c_4.$
By Proposition \ref{absset}, there exists $t_0=t_0(B)\ge 0$, such that for all $t\ge t_0$,
\[
\Psi(\zeta(t)) \le \sup_{\zeta\in \mathcal{H}_0}\{ \Psi(\zeta):\nu_1\|\zeta\|^2_{\mathcal{H}_0} \le 2C \}.
\]
Thus, there is $R_0>0$ such that, for all $t\geq t_0$ and for all $\zeta_0\in \mathcal{H}_0$ with $\|\zeta_0\|_{\mathcal{H}_0} \le \widetilde R$,
\begin{equation}  \label{semiflow-absorbing}
\|S_\alpha(t)\zeta_0\|_{\mathcal{H}_0} \leq R_0.
\end{equation}

By definition, the set $\mathcal{B}_0$ in (\ref{abs-set}) is closed and bounded in $\mathcal{H}_0$. 
The inequality in (\ref{semiflow-absorbing}) implies that $\mathcal{B}_0$ is absorbing: given any nonempty bounded subset $B$ of $\mathcal{H}_0$, there is a $t_0\geq 0$ depending on $B$ in which, for all $t\geq t_0$, $S_\alpha(t)B\subseteq \mathcal{B}_0$. 
Consequently, since $\mathcal{B}_0$ is bounded, $\mathcal{B}_0$ is also positively invariant under the semiflow $S_\alpha$.
This completes the proof of the dissipation of $S_\alpha$ for all $\alpha\in[0,1]$.
\end{proof}

\subsection{Global attractors}  \label{s:global}

In this section we prove

\begin{theorem}  \label{t:global}
The semiflow $S_\alpha$ generated by the mild solutions of Problem {\textbf{P}}$_\alpha$, $\alpha\in[0,1]$, admits a global attractor $\mathcal{A}_{\alpha}$ in $\mathcal{H}_0$. 
The global attractor is invariant under the semiflow $S_\alpha$ (both positively and negatively) and attracts all nonempty bounded subsets of $\mathcal{H}_0$; precisely, 
\begin{enumerate}
\item For each $t\geq 0$, $S_\alpha(t)\mathcal{A}_{\alpha}=\mathcal{A}_{\alpha}$, and 
\item For every nonempty bounded subset $B$ of $\mathcal{H}_0$,
\[
\lim_{t\rightarrow\infty}{\rm{dist}}_{\mathcal{H}_0}(S_\alpha(t)B,\mathcal{A}_{\alpha}):=\lim_{t\rightarrow\infty}\sup_{\zeta\in B}\inf_{\xi\in\mathcal{A}_{\alpha}}\|S_\alpha(t)\zeta-\xi\|_{\mathcal{H}_0}=0.
\]
\end{enumerate}
The global attractor is unique and given by 
\[
\mathcal{A}_{\alpha}:=\omega(\mathcal{B}_0)=\bigcap_{s\geq 0}{\overline{\bigcup_{t\geq s} S_\alpha(t)\mathcal{B}_0}}^{\mathcal{H}_0}.
\]
Furthermore, $\mathcal{A}_{\alpha}$ is the maximal compact invariant subset in $\mathcal{H}_0$.
\end{theorem}

In the previous section it was shown that the semiflow generated by the mild solutions of Problem {\textbf{P}}$_\alpha$, $\alpha\in[0,1],$ admits a bounded absorbing set $\mathcal{B}_0$ in $\mathcal{H}_0$.
According to the standard references on the asymptotic behavior of dissipative dynamical systems (cf. e.g. \cite{Babin&Vishik92,Temam88}), it suffices to show that the semiflows are precompact, or decompose into decaying to zero and precompact parts. 
We already know the following: For each $\alpha\in[0,1],$ and for any $t_0>0$, the set 
\begin{equation}  \label{some}
\bigcup_{t\ge t_0}S_\alpha(t)\mathcal{B}_0 \quad \text{is bounded in $D(A_{\alpha}^{1/\gamma})$,}
\end{equation}
where $\gamma = 1$ if $\alpha > 0$ and $\gamma = 2$ if $\alpha = 0$.
Equation \eqref{some} can be deduced from Lemma \ref{t:abs-set} and Remark \ref{smoothing bound final}.
However, due to the complicated nature of the domain of the operator $A_\alpha,$ we do not know that the solutions to Problem {\textbf{P}}$_\alpha$ regularize into precompact trajectories (see (\ref{regularity})).
The reason for this being the first component of the trajectory does not necessarily regularize into $H^{1+\rho}(\Omega)$, for some $\rho>0$, which, of course, is compact in $H^1(\Omega).$
Keep in mind the operator $A_\alpha$ does not possess compact resolvent, and, additionally, we do not have a suitable $H^2$-elliptic regularity estimate to aid in producing additional smoothness.

To partially remedy this situation, we prove the following:
\begin{lemma}
Let $W_n = (w_{1,n},w_{2,n},w_{3,n},w_{4,n})$ be a bounded sequence in $D(A_\alpha^{1/\gamma})$, where $\gamma = 1$ if $\alpha > 0$ and $\gamma = 2$ if $\alpha = 0$.
Then the sequences $\{w_{2,n}\}$ and $\{w_{4,n}\}$ are precompact in $L^2(\Omega)$ and $L^2(\Gamma)$, respectively.
Moreover, we can identify $w_{4,n}$ with the trace of $w_{2,n}$, i.e.~$w_{4,n} = w_{2,n\mid\Gamma}$.
\end{lemma}

\begin{proof}
We will consider only the case when $\alpha = 0$ and $\gamma = 2$, the other case being much simpler.
Let $\Phi_n = (\phi_{1,n},\phi_{2,n},\phi_{1,n\mid\Gamma},\phi_{4,n})$ be a bounded sequence in $\s{H}_0$ such that $A_0^{1/2}W_n = \Phi_n$.
We write
\begin{equation}
\label{W_n frac power}
W_n = A_0^{-1/2}\Phi_n = \frac{1}{\pi i}\int_0^\infty \lambda^{-1/2}R(\lambda,A_0)\Phi_n d\lambda.
\end{equation}
For $\lambda > 0$, consider $U_n(\lambda) = R(\lambda,A_0)\Phi_n$.
Then by solving the resolvent equation we see that $U_n(\lambda) = (u_n,v_n,u_{n\mid\Gamma},v_{n\mid\Gamma})$ where
$$
u_n = \frac{v_n +\phi_{1,n}}{\lambda} 
$$
and $v_n$ solves, for all $f \in H^1(\Omega) \cap H^1(\Gamma)$,
\begin{multline} \label{resolvent eq}
\left(\lambda + 1 + \frac{1}{\lambda}\right)\left\{ \langle v_n,f \rangle_{L^2(\Omega)} + \langle v_n,f \rangle_{L^2(\Gamma)}\right\}
+ \left(\omega + \frac{1}{\lambda}\right) \langle\nabla v_n,\nabla f \rangle_{L^2(\Omega)}
+ \frac{1}{\lambda}\langle \nabla_\Gamma v_n,\nabla_\Gamma f \rangle_{L^2(\Gamma)}
\\
= -\frac{1}{\lambda} \langle \nabla \phi_{1,n},\nabla f \rangle_{L^2(\Omega)}
-\frac{1}{\lambda} \langle \phi_{1,n},f \rangle_{L^2(\Omega)}
-\frac{1}{\lambda}\langle\nabla_\Gamma \phi_{1,n},\nabla_\Gamma f \rangle_{L^2(\Gamma)}
-\frac{1}{\lambda}\langle \phi_{1,n}, f \rangle_{L^2(\Gamma)}
\\
+ \langle \phi_{2,n},f \rangle_{L^2(\Omega)} + \langle \phi_{4,n},f \rangle_{L^2(\Gamma)}.
\end{multline}
From \eqref{resolvent eq} by taking $f = v_n$ we can deduce that,
\begin{align} \label{v_n energy} 
\|v_n\|_{H^1(\Omega)},\|v_n\|_{H^1(\Gamma)} & \leq \|\Phi_n\|_{\s{H}_0}, \\
\label{v_n H^1} \|v_n\|_{H^1(\Omega)} & \leq \frac{C}{\lambda^{1/2}}\|\Phi_n\|_{\s{H}_0}, \\
\label{v_n L^2} \|v_n\|_{L^2(\Omega)} & \leq \frac{1}{\lambda}\|\Phi_n\|_{\s{H}_0}.
\end{align}
Using \eqref{v_n H^1} and \eqref{v_n L^2}, by interpolation we have, for all $\theta \in (0,1)$,
\begin{equation} \label{v_n H theta}
\|v_n\|_{H^\theta(\Omega)} \leq C\|v_n\|_{H^1(\Omega)}^{1-\theta}\|v_n\|_{L^2(\Omega)}^\theta \leq C\lambda^{-(\theta+1)/2}\|\Phi_n\|_{\s{H}_0},
\end{equation}
and from the trace theorem we have, for all $\theta \in (\frac{1}{2},1)$, 
\begin{equation} \label{v_n H theta trace}
\|v_n\|_{H^{\theta-1/2}(\Gamma)} \leq C\lambda^{-(\theta+1)/2}\|\Phi_n\|_{\s{H}_0}.
\end{equation}
Now writing $v_n = v_n(\lambda)$ to denote the dependence on $\lambda$, we can write using \eqref{W_n frac power} that
\begin{equation} \label{w2 and w4}
w_{2,n} = \frac{1}{\pi i}\int_0^\infty \lambda^{-1/2}v_n(\lambda) d\lambda \quad \text{and} \quad w_{4,n} = \frac{1}{\pi i}\int_0^\infty \lambda^{-1/2}R(\lambda,A_0)v_{n\mid\Gamma}(\lambda) d\lambda,
\end{equation}
so long as these integrals converge in $L^2(\Omega)$ and $L^2(\Gamma)$, respectively.
Indeed, by using \eqref{v_n H theta}, \eqref{v_n H theta trace}, and \eqref{v_n energy} we get, for all $\theta \in (0,1)$,
\begin{equation*}
\|w_{2,n}\|_{H^\theta(\Omega)} \leq \int_0^1 \lambda^{-1/2}\|\Phi_n\|_{\s{H}_0}d\lambda + C\int_1^\infty \lambda^{-\theta/2 - 1}\|\Phi_n\|_{\s{H}_0}d\lambda \leq C_\theta \|\Phi_n\|_{\s{H}_0},
\end{equation*}
and, for all $\theta \in (\frac{1}{2},1)$, 
\begin{equation*}
\|w_{4,n}\|_{H^{\theta-1/2}(\Gamma)} \leq \int_0^1 \lambda^{-1/2}\|\Phi_n\|_{\s{H}_0}d\lambda + C\int_1^\infty \lambda^{-\theta/2 - 1}\|\Phi_n\|_{\s{H}_0}d\lambda \leq C_\theta \|\Phi_n\|_{\s{H}_0}.
\end{equation*}
By taking any $\theta \in (\frac{1}{2},1)$, this proves that the integrals in \eqref{w2 and w4} converge, and moreover it allows us to identify $w_{4,n} = w_{2,n\mid\Gamma}$.
Furthermore, since $H^\theta$ is compactly embedded in $L^2$ for all $\theta > 0$, the compactness claim is proved.
\end{proof}

\begin{proof}[Proof of Theorem \ref{t:global}]
To show the existence of a family of global attractors in $\mathcal{H}_0$, we appeal to \cite[Theorem A.1]{DellOro_Pata_2011} (or also see, for example, \cite[Section 3.2]{Hale88} or \cite[Section 2.7]{Milani&Koksch05}).
Accordingly, we only to show that there exists a time $T > 0$, a constant $0 < \nu < 1$ and a pre-compact pseudo-metric $\eta$ on $\s{B}_0$, where $\s{B}_0$ is the bounded absorbing set from Lemma \ref{t:abs-set}, such that
\begin{equation} \label{attractor condition}
\|S_\alpha(T)\zeta_1 - S_\alpha(T)\zeta_2\|_{\s{H}_0} \leq \nu\|\zeta_1 - \zeta_2\|_{\s{H}_0} + \eta(\zeta_1,\zeta_2) \quad \forall \zeta_1,\zeta_2 \in \s{B}_0.
\end{equation}
Let $\zeta_1 = (u_1,v_1,\gamma_1,\delta_1)$ and $\zeta_2 = (u_2,v_2,\gamma_2,\delta_2)$ in $\s{B}_0$ be given, and let $\zeta = \zeta_1 - \zeta_2 = (u,v,\gamma,\delta)$.
Then $S_\alpha(t)\zeta = (u(t),u_t(t),u(t)|_\Gamma,u_t(t)|_\Gamma)$ is the solution of
\begin{equation} \label{eq:differences}
\left\{
\begin{array}{c}
\begin{array}{rcl}
u_{tt} - \omega\Delta u_t + u_t - \Delta u = \widetilde{f} & \text{in} & \Omega \times (0,T) \\
u_{tt} + \partial_n (u+\omega u_t) + u_t - \Delta_\Gamma ( \alpha \omega u_t + u) + u = \widetilde{g} & \text{on} & \Gamma \times (0,T) \\
\end{array}\\
(u(0),u_t(0),u_{\mid\Gamma}(0),u_{t\mid\Gamma}(0)) = \zeta,
\end{array} \right.
\end{equation}
where $\widetilde{f} = f(u_2(t))-f(u_1(t))$ and $\widetilde{g} = g(u_2(t)) - g(u_1(t))$.
For now we will assume in addition that $\zeta_1$ and $\zeta_2$ are smooth, i.e.~contained in $D(A_\alpha)$.
Then since $S_\alpha(t)\zeta_1$ and $S_\alpha(t)\zeta_2$ are bounded trajectories in $\s{H}_0$, and since $f:H^1(\Omega) \to L^2(\Omega)$ and $g:H^1(\Gamma) \to L^2(\Gamma)$ are locally Lipschitz, we may replace them by globally Lipschitz truncations;
hence, by monotone operator theory, \eqref{eq:differences} generates a nonlinear strongly continuous semigroup (for details, see, for instance, \cite{Graber_Said-Houari_2012}).
Thus for initial data in $D(A_\alpha)$ we have $\zeta(t) = S_\alpha(t)\zeta_1 - S_\alpha(t)\zeta_2$ also in $D(A_\alpha)$ for all $t > 0$, and in particular the solution is classical.

Define the functional on trajectories $\zeta(t)=(u(t),u_{t}(t),u_{\mid\Gamma}(t),u_{t\mid\Gamma}(t))$ for $t\ge 0$ by,
\begin{align*}
I(t) & := \frac{\omega}{2}\|\nabla u(t)\|_{L^2(\Omega)}^2 + \frac{\alpha \omega}{2}\|\nabla_\Gamma u(t)\|_{L^2(\Gamma)}^2 + \|u(t)\|_{L^2(\Omega)}^2 + \|u(t)\|_{L^2(\Gamma)}^2  \notag \\ 
& + \langle u_t(t),u(t)\rangle_{L^2(\Omega)} + \langle u_t(t),u(t) \rangle_{L^2(\Gamma)}.
\end{align*}
By using $u$ as a multiplier in \eqref{eq:differences} and integrating by parts, we get, for almost all $t\ge0,$
\begin{equation*}
\frac{d}{dt}I + \|u\|_{H^1(\Omega)}^2 + \|u\|_{H^1(\Gamma)}^2
 = \left\langle \widetilde{f},u \right\rangle_{L^2(\Omega)} + \left\langle \widetilde{g},u \right\rangle_{L^2(\Gamma)} + \|u_t\|_{L^2(\Omega)}^2 + \|u_t\|_{L^2(\Gamma)}^2.
\end{equation*}
We define also the following energy functional,
\begin{align*}
E(t) & := \frac{1}{2}\|S_\alpha(t)\zeta_1 - S_\alpha(t)\zeta_2\|_{\s{H}_0}^2  \notag \\ 
& = \frac{1}{2}\|u(t)\|_{H^1(\Omega)}^2 + \frac{1}{2}\|u(t)\|_{H^1(\Gamma)}^2
+ \frac{1}{2}\|u_t(t)\|_{L^2(\Omega)}^2 + \frac{1}{2}\|u_t(t)\|_{L^2(\Gamma)}^2.
\end{align*}
We then set $L = E + \epsilon I$, for some $\epsilon > 0$ to be determined.
Recalling the formal energy identity  \eqref{energy-p}, we can write
\begin{align}
\frac{d}{dt}L & + \omega \|\nabla u_t\|_{L^2(\Omega)}^2 + \alpha \omega \|\nabla_\Gamma u_t\|_{L^2(\Gamma)} + \|u_t\|_{L^2(\Omega)}^2 + \|u_t\|_{L^2(\Gamma)}^2 + \epsilon \|u\|_{H^1(\Omega)}^2 + \epsilon \|u\|_{H^1(\Gamma)}^2  \notag \\
& = \epsilon \|u_t\|_{L^2(\Omega)}^2 + \epsilon \|u_t\|_{L^2(\Gamma)}^2 + \left\langle u_t,\widetilde{f} \right\rangle_{L^2(\Omega)} + \left\langle u_t,\widetilde{g} \right\rangle_{L^2(\Gamma)} + \epsilon \left\langle u,\widetilde{f} \right\rangle_{L^2(\Omega)} + \epsilon \left\langle u,\widetilde{g} \right\rangle_{L^2(\Gamma)}.  \notag
\end{align}
For $\epsilon >0$ small enough we can write, for each $t\ge0$,
\begin{equation*}
L \leq \frac{3}{4}\|u\|_{H^1(\Omega)}^2 + \frac{3}{4}\|u\|_{H^1(\Gamma)}^2 + \|u_t\|_{L^2(\Omega)}^2 +\|u_t\|_{L^2(\Gamma)}^2,
\end{equation*}
from which follows,
\begin{align}
\frac{d}{dt}L & + \epsilon L + \omega \|\nabla u_t\|_{L^2(\Omega)}^2 + \alpha \omega \|\nabla_\Gamma u_t\|_{L^2(\Gamma)}  \notag \\
& + (1-2\epsilon)\|u_t\|_{L^2(\Omega)}^2 + (1-2\epsilon)\|u_t\|_{L^2(\Gamma)}^2 + \frac{\epsilon}{4} \|u\|_{H^1(\Omega)}^2 + \frac{\epsilon}{4} \|u\|_{H^1(\Gamma)}^2 \notag \\
& \leq \langle u_t,\widetilde{f}\rangle_{L^2(\Omega)} + \langle u_t,\widetilde{g} \rangle_{L^2(\Gamma)} + \epsilon \langle u,\widetilde{f}\rangle_{L^2(\Omega)} + \epsilon \langle u,\widetilde{g} \rangle_{L^2(\Gamma)}.  \notag
\end{align}
Recalling that $f$ and $g$ are locally Lipschitz and that $u_1,u_2$ are bounded in $H^1(\Omega)$ and $H^1(\Gamma)$, we have,
\begin{equation*}
\left\|\widetilde{f}\right\|_{L^2(\Omega)} = \|f(u_1)-f(u_2)\|_{L^2}(\Omega) \leq C\|u_1-u_2\|_{H^1(\Omega)} = C\|u\|_{H^1(\Omega)},
\end{equation*}
and likewise $\|\widetilde{g}\|_{L^2(\Gamma)} \leq C\|u\|_{H^1(\Gamma)}$.
This is sufficient to obtain, for almost all $t\ge0$,
\begin{align}
& \frac{d}{dt}L + \epsilon L + \omega \|\nabla u_t\|_{L^2(\Omega)}^2 + \alpha \omega \|\nabla_\Gamma u_t\|_{L^2(\Gamma)}  \notag \\
& + (1-2\epsilon)\|u_t\|_{L^2(\Omega)}^2 + (1-2\epsilon)\|u_t\|_{L^2(\Gamma)}^2 + \frac{\epsilon}{8} \|u\|_{H^1(\Omega)}^2 + \frac{\epsilon}{8} \|u\|_{H^1(\Gamma)}^2  \notag \\
& \leq \left\langle u_t,\widetilde{f} \right\rangle_{L^2(\Omega)} + \left\langle u_t,\widetilde{g} \right\rangle_{L^2(\Gamma)} + C\epsilon \|u\|^2_{L^2(\Omega)} + C\epsilon \|u\|^2_{L^2(\Gamma)}.  \notag
\end{align}
For the remaining terms, we have the more precise estimates,
\begin{align*}
\left\langle u_t,\tilde{f} \right\rangle_{L^2(\Omega)} 
&\leq \|u_t\|_{L^6(\Omega)}\|f(u_1)-f(u_2)\|_{L^{6/5}(\Omega)}
\\
&\leq \|u_t\|_{L^6(\Omega)}\|u\|_{L^{2}(\Omega)}\left\|1 + |u_1|^2 + |u_2|^2\right\|_{L^3(\Omega)}
\\
&\leq C\|u_t\|_{H^1(\Omega)}\|u\|_{L^{2}(\Omega)} \left( 1 + \|u_1\|_{L^6(\Omega)}^2 + \|u_2\|^2_{L^6(\Omega)} \right)
\\
&\leq C\|u_t\|_{H^1(\Omega)}\|u\|_{L^{2}(\Omega)},
\end{align*}
and similarly, we find
\begin{align*}
\left\langle u_t,\tilde{g} \right\rangle_{L^2(\Gamma)} & \leq \|u_t\|_{L^4(\Gamma)}\|u\|_{L^2(\Gamma)} \left\|1 + |u_1|^{\rho-1} + |u_2|^{\rho-1} \right\|_{L^4(\Gamma)}  \notag \\
&\leq C\|u_t\|_{H^{1/2}(\Gamma)}\|u\|_{L^2(\Gamma)} \left( 1 + \|u_1\|^{\rho-1}_{L^{4(\rho-1)}(\Gamma)} + \|u_2\|^{\rho-1}_{L^{4(\rho-1)}(\Gamma)} \right)  \notag \\ 
& \leq C\|u_t\|_{H^{1}(\Omega)}\|u\|_{L^2(\Gamma)},
\end{align*}
where $H^{1/2}(\Gamma) \hookrightarrow L^4(\Gamma)$ by the Sobolev embedding theorem and $H^{1}(\Omega) \hookrightarrow H^{1/2}(\Gamma)$ by the trace theorem.
Thus, for $\epsilon > 0$ small enough we can conclude,
\begin{equation*}
\frac{d}{dt}L + \epsilon L \leq C\|u\|^2_{L^2(\Omega)} + C \|u\|^2_{L^2(\Gamma)},
\end{equation*}
where upon integration we deduce, for all $t\ge0,$ 
\begin{equation*}
L(t) \leq e^{-\epsilon t}L(0) + C\int_0^t \left( \|u(s)\|^2_{L^2(\Omega)} + \|u(s)\|^2_{L^2(\Gamma)} \right) ds.
\end{equation*}
On the other hand, we have from the very definition of $L$,
\begin{equation*}
\frac{1}{4}\|S_\alpha(t)\zeta_1 - S_\alpha(t)\zeta_2\|_{\s{H}_0}^2 \leq L(t)
\leq (1+\omega)\|S_\alpha(t)\zeta_1 - S_\alpha(t)\zeta_2\|_{\s{H}_0}^2,
\end{equation*}
so it follows that, for all $t\ge0,$
\begin{equation*}
\|S_\alpha(t)\zeta_1 - S_\alpha(t)\zeta_2\|_{\s{H}_0}^2 \leq Ce^{-\epsilon t}\|\zeta\|_{\s{H}_0}^2 + Ct \sup_{s \in [0,t]} \left( \|u_1(s)-u_2(s)\|_{L^2(\Omega)}^2 + \|u_1(s)-u_2(s)\|_{L^2(\Gamma)}^2 \right).
\end{equation*}
Now by continuity of the semigroup, this inequality also holds for any $\zeta_1,\zeta_2 \in \s{B}_0$, since the domain of the generator is dense in $\s{H}_0$.
Moreover, the pseudo-metric
$$
\eta(\zeta_1,\zeta_2) := \sup_{s \in [0,t]} (\|u_1(s)-u_2(s)\|_{L^2(\Omega)}^2 + \|u_1(s)-u_2(s)\|_{L^2(\Gamma)}^2)^{1/2}
$$
is precompact in $\s{B}_0$ since $H^1$ is compact in $L^2$.
Therefore \eqref{attractor condition} is satisfied.
The proof of Theorem \ref{t:global} is now complete.
\end{proof}

The dimension of these global attractors is explored further once we obtain the results in a later section.
The next section contains a short continuity result for the family of global attractors.

\subsection{Continuity properties of the global attractors}

In this section we aim to prove, in a sense to be made more precise shortly, that
\[
\mathcal{A}_\alpha \rightarrow \mathcal{A}_0 \quad \text{in}\quad \mathcal{H}_0 \quad \text{as}\ \quad \alpha \rightarrow 0.
\]
Following, for example \cite[Section 10.8]{Robinson01}, the perturbation induced by the parameter $\alpha>0$ to Problem {\textbf{P}}$_0$, is termed {\em{regular}} because both classes of Problem {\textbf{P}} lie in the same ``standard energy'' phase space; in particular, the family of global attractors, $\{\mathcal{A}_\alpha\}_{\alpha\in[0,1]}$, lies in $\mathcal{H}_0$.
We will utilize \cite[Theorem 10.16]{Robinson01}.

\begin{proposition}  \label{t:usc}
Assume that for $\ep\in[0,\ep_0)$ the semigroups $S_\ep$ each have admit a global attractor $\mathcal{A}_\ep$ and that there exists a bounded set $X$ such that 
\[
\bigcup_{\ep\in[0,\ep_0)} \mathcal{A}_\ep \subset X.
\]
If in addition the semigroup $S_\ep$ converges to $S_0$ in the sense that, for each $t>0$, $S_\ep(t)x\rightarrow S_0(t)x$ uniformly on bounded subsets $Y$ of the phase space $H$,
\[
\sup_{x\in Y}\|S_\ep(t)x-S_0(t)x\|_H\rightarrow 0 \quad \text{as} \quad \ep\rightarrow 0,
\]
then 
\[
{\rm{dist}}(\mathcal{A}_\ep,\mathcal{A}_0)\rightarrow 0 \quad \text{as} \quad \ep\rightarrow 0.
\]
\end{proposition}

We will show, with a fair amount of ease, that the family of global attractors for Problem {\textbf{P}}$_\alpha$, for all $\alpha\in[0,1]$, are upper-semicontinuous in the topology of $\mathcal{H}_0$.
We now arrive at our first result. 

\begin{lemma}  \label{to-upper-sc}
Let $B$ be a bounded set in $\mathcal{H}_0$ and $T>0$. 
There exists a constant $M=M(\|B\|_{\mathcal{H}_0},T)>0$ such that for all $\zeta_0\in B$ and for all $t\in[0,T]$, there holds, for all $\alpha\in(0,1],$
\begin{align}
\|S_\alpha(t)\zeta_0-S_0(t)\zeta_0\|_{\mathcal{H}_0} \le \sqrt{\alpha}\cdot M.\label{pert-1}
\end{align}
\end{lemma}

\begin{proof}
Let $B$ be a bounded set on $\mathcal{H}_0$, $T>0$, and $\alpha\in(0,1]$.
Let $\zeta_{0} = (u_{0},u_{1},\gamma_{0},\gamma_{1})\in B$. 
For $t>0$, let 
\[
\zeta^1(t)=(u^1(t),u_t^1(t),u_{\mid\Gamma}^1(t),u_{t\mid\Gamma}^1(t)) \quad \text{and} \quad \zeta^2(t)=(u^2(t),u_t^2(t),u_{\mid\Gamma}^2(t),u_{t\mid\Gamma}^2(t)),
\]
denote the corresponding global solutions of Problem {\textbf{P}}$_\alpha$ and Problem {\textbf{P}}$_0$, respectively, on $[0,T]$, both with the (same) initial data $\zeta_{0}$.
For all $t\in(0,T]$, set
\begin{align}
\bar\zeta(t) & := \zeta^1(t) - \zeta^2(t)  \notag \\
& = \left( u^1(t),u_t^1(t),u_{\mid\Gamma}^1(t),u_{t\mid\Gamma}^1(t) \right) - \left( u^2(t),u^2_t(t),u^2_{\mid\Gamma}(t),u^2_{t\mid\Gamma}(t) \right)  \notag \\
& =: \left( \bar u(t),\bar u_t(t),\bar u_{\mid\Gamma}(t),\bar u_{t\mid\Gamma}(t) \right).  \notag
\end{align}
Then $\bar\zeta$ and $\bar u$ satisfy the equations
\begin{equation}  \label{diff-pert}
\left\{ \begin{array}{ll}
\bar{u}_{tt}-\omega\Delta\bar{u}_t+\bar{u}_{t}-\Delta\bar{u}+\bar{u}+f(u^1)-f(u^2)=0 & \text{in}\quad (0,\infty )\times\Omega \\ 
\bar{u}_{tt}+\partial_{\mathbf{n}}(\bar{u}+\omega\bar{u}_t)+\bar{u}_t-\Delta_\Gamma\bar{u}+\bar{u}+g(u^1)-g(u^2)=\alpha\omega\Delta_\Gamma u^1_t & \text{on} \quad (0,\infty)\times \Gamma \\ 
\bar{\zeta}(0)={\mathbf{0}} & \text{at} \quad \{0\}\times \Omega.
\end{array}
\right.
\end{equation}
To obtain \eqref{pert-1}, we now only need to follow the proof of Proposition \ref{t:wk-unique}: after multiplying the equation \eqref{diff-pert}$_1$ by $2\bar u_t$ in $L^2(\Omega)$, we estimate the new product to arrive at, 
\begin{align}
-2\alpha\omega\langle -\Delta_\Gamma u^1_t,\bar u_t \rangle_{L^2(\Gamma)} & \le \alpha\omega\|\nabla_\Gamma u^1_t\|_{L^2(\Gamma)} + \alpha\omega\|\nabla_\Gamma \bar u_t\|_{L^2(\Gamma)}  \notag \\ 
& \le \alpha\cdot Q(\|B\|_{\mathcal{H}_0}) + \alpha\omega\| \bar u_t\|_{H^1(\Gamma)}.  \label{diff-pert-3}
\end{align}
The result readily follows.
This finishes the proof.
\end{proof}

\begin{remark}
The above result \eqref{pert-1} establishes that, on compact time intervals, the difference between trajectories of Problem {\textbf{P}}$_\alpha$, $\alpha\in(0,1]$ and Problem {\textbf{P}}$_0$, originating from the same initial data in $B\subset\mathcal{H}_0$, can be controlled, explicitly, in terms of the perturbation parameter $\alpha.$
\end{remark}

The well-known upper-semicontinuity result in Proposition \ref{t:usc} now follows for our family of global attractors.

\begin{theorem}  \label{t:robustness}
The family of global attractors $\{\mathcal{A}_\alpha\}_{\alpha\in[0,1]}$ is upper-semicontinuous in the topology of $\mathcal{H}_0$; precisely, there is a constant $C>0$ independent of $\alpha$ in which there holds,
\[
{\rm dist}_{\mathcal{H}_0}(\mathcal{A}_\alpha,\mathcal{A}_0) \le \sqrt{\alpha}\cdot C.
\]
\end{theorem}

\subsection{Weak exponential attractors}  \label{s:exp-attr}

This section is motivated by \cite[\S4]{Pata&Zelik06-2}.
In this section we show the existence of a so-called weak exponential attractor. 
We seek a weak exponential attractor for two central reasons.
First, we recall the complicated nature of the domain of the operator associated with the abstract Cauchy problem means the solution operators lack compactness in the standard energy phase space. 
The second reason involves the presence of the dynamic flux term $\partial_{\mathbf{n}}u_t$ in our boundary condition.
Recently, works such as \cite{CGGM10,Frigeri&ShombergXX,Gal07,Gal&Grasselli08,Gal&Shomberg15,GrMS10,Miranville&Zelik05} are able to establish the existence of an exponential attractor through for wave equations with dynamic boundary conditions through the use of suitable $H^2$-elliptic regularity estimates. 
Because of the dynamic flux term, $\partial_{\mathbf{n}}u_t$, such estimates are not available here.

Here we define the space,
\[
\mathcal{H}_{-1} := L^2(\Omega)\times H^{-1}(\Omega)\times L^2(\Gamma)\times H^{-1}(\Gamma),
\]
endowed with the canonical norm, and there holds
\[
\mathcal{H}_0 \hookrightarrow \mathcal{H}_{-1},
\]
with continuous injections.
\begin{theorem}  \label{t:wk-exp-attr}
For each $\alpha\in[0,1]$, the semiflow $S_\alpha$ generated by the weak solutions of Problem \textbf{P}$_\alpha$ admits a weak exponential attractor $\mathcal{M}_{-1,\alpha}$ that satisfies:
\begin{enumerate}
\item $\mathcal{M}_{-1,\alpha}$ is bounded in $\mathcal{H}_0$ and compact in $\mathcal{H}_{-1}$,
\item $\mathcal{M}_{-1,\alpha}$ is positively invariant; i.e., for all $t\ge0,$ $S_\alpha(t)\mathcal{M}_{-1,\alpha}\subseteq\mathcal{M}_{-1,\alpha},$
\item $\mathcal{M}_{-1,\alpha}$ attracts bounded subsets of $\mathcal{H}_0$ exponentially with the metric of $\mathcal{H}_{-1}$; i.e., there exists $\nu>0$ and $Q$ such that, for every bounded subset $B\subset\mathcal{H}_0$ and for all $t\ge0,$
\[
{\rm{dist}}_{\mathcal{H}_{-1}}(S_\alpha(t)B,\mathcal{M}_{-1,\alpha})\le Q(\|B\|_{\mathcal{H}_0})e^{-\nu t}.
\]
\item $\mathcal{M}_{-1,\alpha}$ possesses finite fractal dimension in $\mathcal{H}_{-1}$; i.e.,
\begin{equation}  \label{fdimension}
{\rm{dim}}_F(\mathcal{M}_{-1,\alpha},\mathcal{H}_{-1}):=\limsup_{r\rightarrow 0}\frac{\ln \mu_{\mathcal{H}_{-1}}(\mathcal{M}_{-1,\alpha},r)}{-\ln r}<\infty,
\end{equation}
where $\mu _{\mathcal{H}_{-1}}(\mathcal{M}_{-1,\alpha},r)$ denotes the minimum number of balls of radius $r$ from $\mathcal{H}_{-1}$ required to cover $\mathcal{M}_{-1,\alpha}$.
\end{enumerate}
\end{theorem}

\begin{corollary}
In the topology of $\mathcal{H}_{-1}$, the global attractor possesses finite fractal dimension. 
Indeed, there holds
\begin{equation*}
{\rm{dim}}_F(\mathcal{A}_{\alpha},\mathcal{H}_{-1}) \le {\rm{dim}}_F(\mathcal{M}_{-1,\alpha},\mathcal{H}_{-1}) \le C,
\end{equation*}
for some constant $C>0$, independent of $\alpha$.
\end{corollary}

\begin{remark}
Naturally, after demonstrating the existence of global attractors for Problem {\textbf{P}}$_\alpha$, $\alpha\in[0,1]$, the next question is whether the global attractors are finite dimensional, {\em{in the standard phase space}} $\mathcal{H}_0$. 
Unfortunately, we cannot yet conclude that the fractal dimension of $\mathcal{M}_{-1,\alpha}$ (nor $\mathcal{A}_{\alpha}$) in $\mathcal{H}_0$ is finite.
\end{remark}

The proof of Theorem \ref{t:wk-exp-attr} follows from the application of an abstract result modified only to suit our needs here (for further reference, see for example, \cite{EFNT95,EMZ00,GGMP05}, and also the reference associated with Remark \ref{rem_att}\ below).

\begin{proposition}  \label{abstract1}
Let $H_0$ and $H_{-1}$ be Hilbert spaces such that the embedding $H_0\hookrightarrow H_{-1}$ is compact. 
Let $S=(S(t))_{t\ge0} $ be a semiflow on $H_0$. 
Assume the following hypotheses hold:

\begin{enumerate}
\item[(H1)] There exists a bounded absorbing set $\mathcal{B}_0\subset H_0$ which is positively invariant for $S(t).$ 
More precisely, there exists a time $t_{0}>0$ (possibly depending on the radius of $\mathcal{B}_0$) such that, for all $t\ge t_0,$
\begin{equation*}
S(t)\mathcal{B}_0\subset \mathcal{B}_0.
\end{equation*}

\item[(H2)] There is $t^*\ge t_{0}$ such that the map $S(t^*)$ admits the decomposition, for all $\zeta_{01},\zeta _{02}\in \mathcal{B}_0,$ 
\begin{equation*}
S(t^*)\zeta_{01} - S(t^*)\zeta_{02} = L(\zeta_{01},\zeta_{02}) + K(\zeta_{01},\zeta_{02}),
\end{equation*}
where, for some constants $\kappa=\kappa(t^*) \in (0,\frac{1}{2})$ and $\Lambda=\Lambda(t^*)\ge 0$, the following hold:
\begin{equation}  \label{H2-L}
\|L(\zeta_{01},\zeta_{02})\|_{H_{-1}} \le \kappa \|\zeta_{01}-\zeta_{02} \|_{H_{-1}}
\end{equation}
and
\begin{equation}
\|K(\zeta_{01},\zeta_{02})\|_{H_0} \le \Lambda \|\zeta_{01}-\zeta_{02} \|_{H_{-1}}.  \label{H2-K}
\end{equation}

\item[(H3)] The map
\begin{equation}  \label{to-H3}
(t,\zeta_0)\mapsto S(t)\zeta:[t^*,2t^*]\times \mathcal{B}_0\rightarrow \mathcal{B}_0
\end{equation}
is Lipschitz continuous on $\mathcal{B}_0$ in the topology of $H_{-1}$.
\end{enumerate}

Then the semiflow $S$ admits an exponential attractor $\mathcal{M}_{-1}$ in $\mathcal{B}_0.$
\end{proposition}

We now show that the hypotheses (H2) and (H3) hold for the semiflow $S_\alpha$ generated by the mild solutions of Problem \textbf{P}$_\alpha$. 
Of course, the first hypothesis (H1) was already shown in Lemma \ref{t:abs-set}.
Moving forward, we now show (H2) by making the appropriate ``lower-order'' estimates in the norm of $\mathcal{H}_{-1}.$ 

\begin{lemma} \label{t:to-H2} 
For all $\alpha\in[0,1],$ condition (H2) holds. 
\end{lemma}

\begin{proof}
Let $\alpha\in[0,1]$, and $\zeta_{01} = (u_{01},u_{11},\gamma_{01},\gamma_{11}), \zeta_{02} = (u_{02},u_{12},\gamma_{02},\gamma_{12})\in\mathcal{B}_0$. 
For $t>0$, let 
\[
\zeta^1(t)=(u^1(t),u_t^1(t),u_{\mid\Gamma}^1(t),u_{t\mid\Gamma}^1(t)) \quad \text{and} \quad \zeta^2(t)=(u^2(t),u_t^2(t),u_{\mid\Gamma}^2(t),u_{t\mid\Gamma}^2(t)),
\]
denote the corresponding global solutions of Problem {\textbf{P}}$_\alpha$ on $[0,T]$ with the initial data $\zeta_{01}$ and $\zeta_{02}$, respectively.
For all $t\in(0,T]$, set
\begin{align}
\bar\zeta(t) & := \zeta^1(t) - \zeta^2(t)  \notag \\
& = \left( u^1(t),u_t^1(t),u_{\mid\Gamma}^1(t),u_{t\mid\Gamma}^1(t) \right) - \left( u^2(t),u^2_t(t),u^2_{\mid\Gamma}(t),u^2_{t\mid\Gamma}(t) \right)  \notag \\
& =: \left( \bar u(t),\bar u_t(t),\bar u_{\mid\Gamma}(t),\bar u_{t\mid\Gamma}(t) \right),  \notag
\end{align}
and 
\begin{align}
\bar\zeta_0 & := \zeta_{01} - \zeta_{02}  \notag \\ 
& = \left( u_{01},u_{11},\gamma_{01},\gamma_{11} \right) - \left( u_{02},u_{12},\gamma_{02},\gamma_{12} \right) \notag \\
& = \left( u_{01}-u_{02},u_{11}-u_{12},\gamma_{01}-\gamma_{02},\gamma_{11}-\gamma_{12} \right)  \notag \\
& =: (\bar u_0,\bar u_1,\bar \gamma_0,\bar \gamma_1).  \notag
\end{align}
For each $t\geq 0$, decompose the difference $\bar{\zeta}(t):=\zeta^1(t)-\zeta^2(t)$ with $\bar{\zeta}_{0}:=\zeta_{01}-\zeta_{02}$ as follows:
\begin{align}
\bar{\zeta}(t) & = \left( \bar u(t),\bar u_t(t),\bar u_{\mid\Gamma}(t),\bar u_{t\mid\Gamma}(t) \right)  \notag \\ 
& = \left( \bar v(t),\bar v_t(t),\bar v_{\mid\Gamma}(t),\bar v_{t\mid\Gamma}(t) \right) + \left( \bar w(t),\bar w_t(t),\bar w_{\mid\Gamma}(t),\bar w_{t\mid\Gamma}(t) \right),  \notag \\ 
& = \bar{\varphi}(t)+\bar{\vartheta}(t),  \notag 
\end{align}
where $\bar{\varphi}(t)$ and $\bar{\vartheta}(t)$ are solutions of the problems:
\begin{equation}  \label{diff-decomp-v}
\left\{ \begin{array}{ll}
\bar{v}_{tt}-\omega\Delta\bar{v}_t+\bar{v}_{t}-\Delta\bar{v}+\bar{v}=0 & \text{in}\quad (0,\infty )\times\Omega \\ 
\bar{v}_{tt}+\partial_{\mathbf{n}}(\bar{v}+\omega\bar{v}_t)-\alpha\omega\Delta_\Gamma\bar{v}_t+\bar{v}_t-\Delta_\Gamma\bar{v}+\bar{v}=0 & \text{on} \quad (0,\infty)\times \Gamma \\ 
\bar{\varphi}(0)=\bar\zeta_{0} & \text{at} \quad \{0\}\times \Omega,
\end{array}
\right.
\end{equation}
and 
\begin{equation}  \label{diff-decomp-w}
\left\{ \begin{array}{ll}
\bar{w}_{tt}-\omega\Delta\bar{w}_t+\bar{w}_{t}-\Delta\bar{w}+\bar{w}=f(u^2)-f(u^1) & \text{in}\quad (0,\infty )\times\Omega \\ 
\bar{w}_{tt}+\partial_{\mathbf{n}}(\bar{w}+\omega\bar{w}_t)-\alpha\omega\Delta_\Gamma\bar{w}_t+\bar{w}_t-\Delta_\Gamma\bar{w}+\bar{w}=g(u^2)-g(u^1) & \text{on} \quad (0,\infty)\times \Gamma \\ 
\bar{\vartheta}(0)=\mathbf{0} & \text{at} \quad \{0\}\times \Omega.
\end{array}
\right.
\end{equation}

Let $\ep\in(0,1)$ to be chosen later and let 
\[
\bar\phi(t):=\int_0^t e^{-\ep(t-\tau)}\bar v(\tau)d\tau.
\]
Then $\bar\phi_t+\ep\bar\phi=\bar v$ and $\bar\phi(0)=0.$
Multiply equations (\ref{diff-decomp-v})$_1$--(\ref{diff-decomp-v})$_2$ by $e^{-\ep(t-\tau)}$ and integrate with respect to $\tau$ over $(0,t)$ to find,
\begin{align}
& \bar\phi_{tt}(t) - \omega\Delta\bar\phi_t(t) + \bar\phi_t(t) - \Delta\bar\phi(t) + \bar\phi(t) = 0 \quad \text{in} \quad \Omega,  \label{wkea1}
\end{align}
and
\begin{align}
& \bar\phi_{tt}(t) +\partial_{\mathbf{n}}(\omega\bar\phi_t(t)+\bar\phi(t)) - \alpha\omega\Delta_\Gamma\bar\phi_t(t) + \bar\phi_t(t) - \Delta_\Gamma\bar\phi(t) + \bar\phi(t) = 0 \quad \text{on} \quad \Gamma.  \label{wkea2}
\end{align}
Now we multiply (\ref{wkea1})--\eqref{wkea2} by $\bar\phi_t+\ep\bar\phi$ in $L^2(\Omega)$ to find,
\begin{align}
\frac{d}{dt} & \left\{ \|\bar \phi_t\|^2_{L^2(\Omega)} + 2\ep\langle \bar\phi_t,\bar\phi \rangle_{L^2(\Omega)} + \|\bar\phi\|^2_{H^1(\Omega)} \right.  \notag \\ 
& \left. + \|\bar\phi_t\|^2_{L^2(\Gamma)} + 2\ep\langle \bar\phi_t,\bar\phi \rangle_{L^2(\Gamma)} + \|\bar\phi\|^2_{H^1(\Gamma)} \right\}  \notag \\ 
& + 2\omega\|\nabla\bar\phi_t\|^2_{L^2(\Omega)} + 2(1-\ep)\|\bar\phi_t\|^2_{L^2(\Omega)} + 2\ep\omega\langle \nabla\bar\phi_t,\nabla\bar\phi\rangle_{L^2(\Omega)}  \notag \\ 
& + 2\ep\langle\bar\phi_t,\bar\phi\rangle_{L^2(\Omega)} + 2\ep\|\bar\phi\|^2_{H^1(\Omega)}  \notag \\ 
& + 2\alpha\omega\|\nabla_\Gamma\bar\phi_t\|^2_{L^2(\Gamma)} + 2(1-\ep)\|\bar\phi_t\|^2_{L^2(\Gamma)} + 2\ep\alpha\omega\langle \nabla_\Gamma\bar\phi_t,\nabla_\Gamma\bar\phi\rangle_{L^2(\Gamma)}  \notag \\ 
& + 2\ep\langle \bar\phi_t,\bar\phi\rangle_{L^2(\Gamma)} + 2\ep\|\bar\phi\|^2_{H^1(\Gamma)} = 0.  \label{wkea3}
\end{align}
Since the solutions $\zeta^1$ and $\zeta^2$ satisfy (\ref{regularity}), then $\widehat\varphi=(\bar\phi,\bar\phi_t,\bar\phi_{\mid\Gamma},\bar\phi_{t\mid\Gamma})\in C([0,\infty);\mathcal{H}_{-1})$.
Hence, each term of (\ref{wkea1})--\eqref{wkea2} when taken in the $L^2$ product with $\bar\phi_t+\ep\bar\phi=\bar v\in C([0,\infty);\mathcal{H}_0)$ is well-defined.

Define the following functional for all $\zeta=(u,v,\gamma,\delta)\in\mathcal{H}_0$ and for each $\ep>0$,
\begin{align}
\Phi(\zeta) := & \|\zeta\|^2_{\mathcal{H}_0} + 2\ep\langle u,v \rangle_{L^2(\Omega)} + 2\ep\langle \gamma,\delta \rangle _{L^2(\Gamma)}.  \label{wkea4}
\end{align}
(On trajectories $\widehat\varphi(t):=(\bar\phi(t),\bar\phi_t(t),\bar\phi_{\mid\Gamma}(t),\bar\phi_{t\mid\Gamma}(t))$, $t>0$, we denote $\Phi(\widehat\varphi)$ by $\Phi(t).$)
Observe, with the Cauchy--Schwarz inequality, Poincar\'{e} (\ref{Poincare}), and the trace embedding $H^1(\Omega)\hookrightarrow L^2(\Gamma)$, a straight forward calculation shows there are constants $C_1,C_2>0$, such that, for all $t\ge0,$
\begin{align}
C_1\|\widehat\varphi(t)\|^2_{\mathcal{H}_0} \le \Phi(t) \le C_2\|\widehat\varphi(t)\|^2_{\mathcal{H}_0}.  \label{wkea5}
\end{align}
Moreover, the map $t\mapsto \Phi(t)$ is $C^1([0,\infty)).$

After the two basic estimates, 
\begin{align}
2\ep\omega\langle \nabla\bar\phi_t,\nabla\bar\phi\rangle_{L^2(\Omega)} & \ge -2\omega\ep\cdot\frac{1}{\ep}\|\nabla\bar\phi_t\|^2_{L^2(\Omega)} - 2\omega\ep\cdot \frac{\ep}{4}\|\nabla\bar\phi\|^2_{L^2(\Omega)}, \notag \\ 
& \ge -2\omega\|\nabla\bar\phi_t\|^2_{L^2(\Omega)} - \frac{\omega\ep^2}{2}\|\bar\phi\|^2_{H^1(\Omega)}, \label{be-1}
\end{align}
and
\begin{align}
2\ep\alpha\omega\langle \nabla_\Gamma\bar\phi_t,\nabla_\Gamma\bar\phi\rangle_{L^2(\Gamma)} & \ge -2\alpha^2\omega\ep\cdot\frac{1}{\ep}\|\nabla_\Gamma\bar\phi_t\|^2_{L^2(\Gamma)} - 2\omega\ep\cdot \frac{\ep}{4}\|\nabla_\Gamma\bar\phi\|^2_{L^2(\Gamma)}, \notag \\ 
& \ge -2\alpha\omega\|\nabla_\Gamma\bar\phi_t\|^2_{L^2(\Gamma)} - \frac{\omega\ep^2}{2}\|\bar\phi\|^2_{H^1(\Gamma)}, \label{be-2}
\end{align}
then (\ref{wkea4}) and the identity (\ref{wkea3}) may be written into the differential inequality,
\begin{align}
\frac{d}{dt}\Phi & + 2(1-\ep)\|\bar\phi_t\|^2_{L^2(\Omega)} + 2\ep\langle\bar\phi_t,\bar\phi\rangle_{L^2(\Omega)} + 2\ep\left(1-\frac{\omega\ep}{4}\right)\|\bar\phi\|^2_{H^1(\Omega)}  \notag \\ 
& + 2(1-\ep)\|\bar\phi_t\|^2_{L^2(\Gamma)} + 2\ep\langle \bar\phi_t,\bar\phi\rangle_{L^2(\Gamma)} + 2\ep\left(1-\frac{\omega\ep}{4} \right)\|\bar\phi\|^2_{H^1(\Gamma)}  \notag \\ 
& \le 0.  \label{wkea6}
\end{align}
Hence, for any $\omega\in(0,1]$, there is some $\ep\in(0,1)$ and a $\nu_2>0$ in which (\ref{wkea6}) becomes, for almost all $t\ge0,$
\begin{align}
\frac{d}{dt}\Phi + \nu_2\Phi \le 0.  \label{wkea7}
\end{align}
Integration of (\ref{wkea7}) over $(s,t)$ yields, with the aid of (\ref{wkea5}),
\begin{align}
\|\widehat\varphi(t)\|_{\mathcal{H}_0} \le Ce^{-\nu_2(t-s)/2} \|\widehat\varphi(s)\|_{\mathcal{H}_0},  \label{wkea8}
\end{align}
where
\[
\widehat\varphi(t)=(\bar\phi(t),\bar\phi_t(t),\bar\phi_{\mid\Gamma}(t),\bar\phi_{t\mid\Gamma}(t)).
\]
Since
\begin{align}
\|\widehat\varphi(t)\|_{\mathcal{H}_0} & = \int_0^t e^{-\ep(t-\tau)} \|\bar\varphi(\tau)\|_{\mathcal{H}_0} d\tau  \notag \\ 
& = \|\bar\varphi(t)\|_{\mathcal{H}_{-1}},
\end{align}
(\ref{wkea8}) becomes, in the limit $s\rightarrow0^+$,
\begin{align}
\|\bar\varphi(t)\|_{\mathcal{H}_{-1}} \le Ce^{-\nu_2 t/2} \|\bar\zeta_0\|_{\mathcal{H}_{-1}}.  \notag
\end{align}
Thus, (\ref{H2-L}) holds for with $\kappa = Ce^{-\nu_2 t^*/2}$ for any fixed $t^*>\max\{t_0,\frac{2}{\nu_2}\ln(2C)\}$.

It remains to show that (\ref{H2-K}) holds to complete the proof.
This time, transform the system (\ref{diff-decomp-w}) with 
\[
\bar\theta(t):=\int_0^t e^{-\ep(t-\tau)}\bar w(\tau) d\tau,
\]
for some (new) $\ep\in(0,1)$ to be determined later below.
Again, $\bar\theta_t+\ep\bar\theta=\bar w$ and $\bar\theta(0)=0.$
Multiply equations (\ref{diff-decomp-w})$_1$--(\ref{diff-decomp-w})$_2$ by $e^{-\ep(t-\tau)}$ and integrate with respect to $\tau$ over $(0,t)$ to find,
\begin{align}
& \bar\theta_{tt}(t) - \omega\Delta\bar\theta_t(t) + \bar\theta_t(t) - \Delta\bar\theta(t) + \bar\theta(t)  \notag \\ 
& + \int_0^t e^{-\ep(t-\tau)} \left( f(u^1(\tau)) - f(u^2(\tau)) \right) d\tau = 0 \quad \text{in} \quad \Omega,  \label{wkea11}
\end{align}
and
\begin{align}
& \bar\theta_{tt}(t) + \partial_{\mathbf{n}}(\omega\bar\theta_t(t) + \bar\theta(t)) - \alpha\omega\Delta_\Gamma\bar\theta_t(t) + \bar\theta_t(t) - \Delta_\Gamma\bar\theta(t) + \bar\theta(t)  \notag \\ 
& + \int_0^t e^{-\ep(t-\tau)} \left( g(u^1(\tau)) - g(u^2(\tau)) \right) d\tau = 0 \quad \text{on} \quad \Gamma.  \label{wkea12}
\end{align}
As before, multiplication of (\ref{wkea11})--\eqref{wkea12} in $L^2(\Omega)$ by $\bar\theta_t+\ep\bar\theta$ is defined.
For notational ease, we proceed to denote 
\[
\mathbb{F}(t) := \int_0^t e^{-\ep(t-\tau)} f(u^1(\tau))-f(u^2(\tau)) d\tau,
\]
and
\[
\mathbb{G}(t) := \int_0^t e^{-\ep(t-\tau)} g(u^1(\tau))-g(u^2(\tau)) d\tau.
\]
Now observe, 
\begin{align}
& \left\langle \mathbb{F}(t),\bar\theta_t \right\rangle_{L^2(\Omega)} = \frac{d}{dt} \left\langle \mathbb{F}(t), \bar\theta \right\rangle_{L^2(\Omega)} - \left\langle f(u^1)-f(u^2),\bar\theta \right\rangle_{L^2(\Omega)} + \ep\left\langle \mathbb{F}(t), \bar\theta \right\rangle_{L^2(\Omega)},  \label{diff-help-1}
\end{align}
and similarly,
\begin{align}
& \left\langle \mathbb{G}(t),\bar\theta_t \right\rangle_{L^2(\Gamma)} = \frac{d}{dt} \left\langle \mathbb{G}(t), \bar\theta \right\rangle_{L^2(\Gamma)} - \left\langle g(u^1)-g(u^2),\bar\theta \right\rangle_{L^2(\Gamma)} + \ep\left\langle \mathbb{G}(t), \bar\theta \right\rangle_{L^2(\Gamma)}.  \label{diff-help-2}
\end{align}
Hence, here we obtain the differential identity,  
\begin{align}
\frac{d}{dt} & \left\{ \|\bar \theta_t\|^2_{L^2(\Omega)} + 2\ep\langle \bar\theta_t,\bar\theta \rangle_{L^2(\Omega)} + \|\bar\theta\|^2_{H^1(\Omega)} \right.  \notag \\ 
& + \|\bar\theta_t\|^2_{L^2(\Gamma)} + 2\ep\langle \bar\theta_t,\bar\theta \rangle_{L^2(\Gamma)} + \|\bar\theta\|^2_{H^1(\Gamma)}  \notag \\ 
& + \left. 2\left\langle \mathbb{F}(t), \bar\theta \right\rangle_{L^2(\Omega)} + 2\left\langle \mathbb{G}(t), \bar\theta \right\rangle_{L^2(\Gamma)} \right\}  \notag \\ 
& + 2\omega\|\nabla\bar\theta_t\|^2_{L^2(\Omega)} + 2(1-\ep)\|\bar\theta_t\|^2_{L^2(\Omega)} + 2\ep\omega\langle \nabla\bar\theta_t,\nabla\bar\theta\rangle_{L^2(\Omega)}  \notag \\ 
& + 2\ep\langle\bar\theta_t,\bar\theta\rangle_{L^2(\Omega)} + 2\ep\|\bar\theta\|^2_{H^1(\Omega)}  \notag \\ 
& + 2\alpha\omega\|\nabla_\Gamma\bar\theta_t\|^2_{L^2(\Gamma)} + 2(1-\ep)\|\bar\theta_t\|^2_{L^2(\Gamma)} + 2\ep\alpha\omega\langle \nabla_\Gamma\bar\theta_t,\nabla_\Gamma\bar\theta\rangle_{L^2(\Gamma)}  \notag \\ 
& + 2\ep\langle \bar\theta_t,\bar\theta\rangle_{L^2(\Gamma)} + 2\ep\|\bar\theta\|^2_{H^1(\Gamma)}  \notag \\ 
& + 4\ep\left\langle \mathbb{F}(t), \bar\theta \right\rangle_{L^2(\Omega)} + 4\ep\left\langle \mathbb{G}(t), \bar\theta \right\rangle_{L^2(\Gamma)}  \notag \\ 
& = 2\left\langle f(u^2)-f(u^1),\bar\theta \right\rangle_{L^2(\Omega)} + 2\left\langle g(u^2)-g(u^1),\bar\theta \right\rangle_{L^2(\Gamma)}.  \label{wkea13}
\end{align}

Define the functional, for all $\zeta=(u,v,\gamma,\delta)\in\mathcal{H}_0$ and for each $\ep>0$,
\begin{align}
\Theta(\zeta) := \Phi(\zeta) & + 2\left\langle \mathbb{F}, \bar\theta \right\rangle_{L^2(\Omega)} + 2\left\langle \mathbb{G}, \bar\theta \right\rangle_{L^2(\Gamma)},  \label{wkea14}
\end{align}
where we recall the functional $\Phi$ was defined above in (\ref{wkea4}).
(Again, we will denote $\Theta(\widehat\vartheta)$ on trajectories $\widehat\vartheta(t):=(\bar\theta(t),\bar\theta_t(t),\bar\theta_{\mid\Gamma}(t),\bar\theta_{t\mid\Gamma}(t))$, $t>0$, by $\Theta(t)$; also, the map $t\mapsto \Theta(t)$ is $C^1([0,\infty)).$)
With (\ref{wkea14}) and the two preceeding basic estimates (\ref{be-1}) and (\ref{be-2}), the identity (\ref{wkea13}) becomes the differential inequality, which holds for almost all $t\ge0$,
\begin{align}
\frac{d}{dt} & \Theta + 2(1-\ep)\|\bar\theta_t\|^2_{L^2(\Omega)} + 2\ep\langle\bar\theta_t,\bar\theta\rangle_{L^2(\Omega)} + 2\ep\left(1-\frac{\omega\ep}{4}\right)\|\bar\theta\|^2_{H^1(\Omega)}  \notag \\ 
& + 2(1-\ep)\|\bar\theta_t\|^2_{L^2(\Gamma)} + 2\ep\langle \bar\theta_t,\bar\theta\rangle_{L^2(\Gamma)} + 2\ep\left(1-\frac{\omega\ep}{4} \right)\|\bar\theta\|^2_{H^1(\Gamma)}  \notag \\ 
& + 4\ep\left\langle \mathbb{F}(t), \bar\theta \right\rangle_{L^2(\Omega)} + 4\ep\left\langle \mathbb{G}(t), \bar\theta \right\rangle_{L^2(\Gamma)}  \notag \\ 
& \le \left\langle f(u^2)-f(u^1),\bar\theta \right\rangle_{L^2(\Omega)} + \left\langle g(u^2)-g(u^1),\bar\theta \right\rangle_{L^2(\Gamma)}.  \label{wkea15}
\end{align}
Following the estimates given above in (\ref{cde-3}) and (\ref{cde-4}), and applying the uniform bound (\ref{wkest-1}) on the solutions $\zeta^1(t),\zeta^2(t)\in \mathcal{H}_0$, we readily find the following hold for $\alpha\in(0,1]$, for all $\ep\in(0,1),$
\begin{align}
\left|\left\langle f(u^1) - f(u^2), \bar\theta \right\rangle_{L^2(\Omega)}\right| & \le 2\ell_1 \|\bar u\|_{L^6(\Omega)} \left( 1+\|u^1\|^2_{L^3(\Omega)} + \|u^2\|^2_{L^3(\Omega)} \right)\|\bar\theta\|_{L^6(\Omega)}  \notag \\ 
& \le \ep\|\bar\theta\|^2_{H^1(\Omega)} + C_\ep(\|\bar\zeta_0\|_{\mathcal{H}_0}),  \label{wkea151}
\end{align}
and
\begin{align}
\left|\left\langle g(u^1) - g(u^2), \bar\theta \right\rangle_{L^2(\Gamma)}\right| & \le \ep\|\bar\theta\|^2_{H^1(\Gamma)} + C_\ep(\|\bar\zeta_0\|_{\mathcal{H}_0}).  \label{wkea152}
\end{align}
Combining (\ref{wkea15})--(\ref{wkea152}), we find, fixing sufficiently small $\ep\in(0,1),$ there exists $\nu_3>0$ such that, 
\begin{align}
\frac{d}{dt} & \Theta + \nu_3\Theta \le C(\|\bar\zeta_0\|_{\mathcal{H}_0}).  \label{wkea153}
\end{align}
Integrating (\ref{wkea153}) over $(0,t)$ yields, for all $t\ge0,$
\begin{align}
\Theta(t) & \le e^{-\nu_3t}\left( \Theta(0)-C(\|\bar\zeta_0\|_{\mathcal{H}_0}) \right) + C(\|\bar\zeta_0\|_{\mathcal{H}_0})  \notag \\ 
& \le \Theta(0) + C(\|\bar\zeta_0\|_{\mathcal{H}_0}).  \label{wkea154}
\end{align}
Rewriting inequality (\ref{wkea154}) with (\ref{wkea14}) and (\ref{wkea4})--(\ref{wkea5}) shows, 
\begin{align}
\|\widehat\vartheta(t)\|^2_{\mathcal{H}_0} & \le C_2\|\widehat\vartheta(0)\|^2_{\mathcal{H}_0} - \frac{2}{C_1}\left( \left\langle \mathbb{F}(t), \bar\theta(t) \right\rangle_{L^2(\Omega)} + \left\langle \mathbb{G}(t), \bar\theta(t) \right\rangle_{L^2(\Gamma)} \right) + C(\|\bar\zeta_0\|_{\mathcal{H}_0}).  \label{wkea155}
\end{align}
Estimating the two remaining products with (\ref{wkea151}) and (\ref{wkea152}) (recall, $t\in[0,T]$ and now $\ep$ is fixed),
\begin{align}
2\left\langle \mathbb{F}(t), \bar\theta(t) \right\rangle_{L^2(\Omega)} & = 2\int_0^t e^{-\ep(t-\tau)} \left\langle f(u^1(\tau))-f(u^2(\tau)), \bar\theta(t) \right\rangle_{L^2(\Omega)} d\tau  \notag \\ 
& \le C(\|\bar\zeta_0\|_{\mathcal{H}_0}) \|\bar\theta(t)\| e^{-\ep t}\int_0^T e^{\ep\tau} d\tau  \notag \\ 
& \le \eta\|\bar\theta\|^2_{H^1(\Omega)} + \frac{1}{4\eta} C(\|\bar\zeta_0\|_{\mathcal{H}_0},T),  \label{wkea16}
\end{align}
for all $\eta>0.$
Similarly, for all $\eta>0,$
\begin{align}
2\left\langle \mathbb{G}(t), \bar\theta(t) \right\rangle_{L^2(\Gamma)} & \le \eta\|\bar\theta\|^2_{H^1(\Gamma)} + \frac{1}{4\eta} C(\|\bar\zeta_0\|_{\mathcal{H}_0},T).  \label{wkea17}
\end{align}
Recalling (\ref{diff-decomp-w})$_3$ and applying (\ref{wkea16})--(\ref{wkea17}) to (\ref{wkea155}), we now find that there holds, for all $t\in[0,T],$
\begin{align}
\|\widehat\vartheta(t)\|^2_{\mathcal{H}_0} & \le C(\|\bar\zeta_0\|_{\mathcal{H}_0},T).  \label{wkea18}
\end{align}
The estimate (\ref{wkea18}) together with the continuous embedding $\mathcal{H}_0\hookrightarrow\mathcal{H}_{-1}$ establishes (\ref{H2-K}).
This concludes the proof.
\end{proof}

\begin{lemma} \label{t:to-H3} 
For all $\alpha\in[0,1],$ condition (H3) holds. 
\end{lemma}

\begin{proof}
Let $\alpha\in[0,1]$, $\zeta_{01},\zeta_{02}\in\mathcal{B}_0$, and $t_1,t_2\in[t^*,2t^*]$.
In the norm of $\mathcal{H}_{-1}$, we calculate
\begin{align}
& \|S_\alpha(t_1)\zeta_{01}-S_\alpha(t_2)\zeta_{02}\|_{\mathcal{H}_{-1}}  \notag \\ 
& \le \|S_\alpha(t_1)\zeta_{01}-S_\alpha(t_1)\zeta_{02}\|_{\mathcal{H}_{-1}} + \|S_\alpha(t_1)\zeta_{02}-S_\alpha(t_2)\zeta_{02}\|_{\mathcal{H}_{-1}}.  \label{h3-1}
\end{align}
The first term on the right-hand side of (\ref{h3-1}) is bounded, uniformly in $t$ on compact intervals, by (\ref{contdep}).
We will deal with the remaining term only in the Gevrey case when $\alpha=0$ since the argument for the analytic case $\alpha\in(0,1]$ is similar.
For this, recall \eqref{eq:gevrey estimate} from Theorem \ref{t:sg-theory}; in the norm of $\mathcal{H}_0$, there is a constant $C>0$ such that there holds, for all $t>0$,
\begin{align}
\|\partial_t\zeta(t)\|_{\mathcal{H}_0} & \le \left\| \frac{d}{dt}e^{A_0t} \right\|_{\mathcal{L(H}_0)}\|\zeta_0\|_{\mathcal{H}_0} + \left\| \mathcal{F}(\zeta(t)) \right\|_{\mathcal{H}_0}  \notag \\ 
& \le \frac{C}{t^2}\|\zeta_0\|_{\mathcal{H}_0} + \left\| \mathcal{F}(\zeta(t)) \right\|_{\mathcal{H}_0}.
\end{align}
Because $t_1,t_2\in[t^*,2t^*]$ and $\zeta_{02}\in\mathcal{B}_0$, we can use \eqref{afix-1}, \eqref{afix-2} and the continuous embedding $\mathcal{H}_0\hookrightarrow\mathcal{H}_{-1}$ to easily show that the term on the right-hand side of \eqref{h3-1} is globally bounded in $\mathcal{H}_{-1}$, whereby, the desired Lipschitz continuity property (\ref{to-H3}) naturally follows for the analytic case, $\alpha\in(0,1]$. 
This concludes the proof. 
\end{proof}

\begin{remark}
Concerning the proof of Lemma \ref{t:to-H3}, when $\alpha>0$ a simpler proof follows directly from the theory of analytic semigroups.
Indeed, since the operator $A_\alpha$ is analytic and the map $t\mapsto\mathcal{F}(\zeta(t))$ is $C([0,\infty);\mathcal{H}_0)$, we may apply standard results from \cite{Pazy83,Renardy&Rogers04}, for example. 
Then, for any $\delta>0,$ 
\begin{equation*}
\partial_t\zeta(t) = (u_{t}(t),u_{tt}(t),u_{t\mid\Gamma}(t),u_{tt\mid\Gamma}(t)) \in C([\delta,\infty);\mathcal{H}_0).
\end{equation*} 
Now with $t_1,t_2\in[t^*,2t^*]$ and $\zeta_{02}\in\mathcal{B}_0$, a simple estimate using \eqref{mild} shows $\partial_t\zeta(t)$ is uniformly bounded in $\mathcal{H}_0$ and the result now follows.
\end{remark}

\begin{remark}  \label{rem_att} 
According to Proposition \ref{abstract1}, the semiflow $S_\alpha(t):\mathcal{H}_0\rightarrow\mathcal{H}_0$ possesses an exponential attractor $\mathcal{M}_{-1,\alpha}\subset \mathcal{B}_0$, which attracts bounded subsets of $\mathcal{B}_0$ exponentially in the topology of $\mathcal{H}_{-1}$. 
Thus, the global attractors $\mathcal{A}_\alpha$ are finite dimensional in the topology of $\mathcal{H}_{-1}.$
In order to show that the attraction property in Theorem \ref{t:wk-exp-attr} part (3) also holds, we appeal to the transitivity of the exponential attraction described in the proceeding proposition.
\end{remark}

The next result is the so-called transitivity property of exponential attraction from \cite[Theorem 5.1]{FGMZ04}).

\begin{proposition}  \label{t:exp-attr}
Let $(\mathcal{X},d)$ be a metric space and let $S_t$ be a semigroup acting on this space such that 
\[
d(S_t m_1,S_t m_2) \leq C e^{Kt} d(m_1,m_2),
\]
for appropriate constants $C$ and $K$. Assume that there exists three subsets $U_1$,$U_2$,$U_3\subset\mathcal{X}$ such that 
\[
{\rm{dist}}_\mathcal{X}(S_t U_1,U_2) \le C_1 e^{-\alpha_1 t} \quad \text{and} \quad {\rm{dist}}_\mathcal{X}(S_t U_2,U_3) \le C_2 e^{-\alpha_2 t}.
\]
Then 
\[
{\rm{dist}}_\mathcal{X}(S_t U_1,U_3) \le C' e^{-\alpha' t},
\]
where $C'=CC_1+C_2$ and $\alpha'=\frac{\alpha_1\alpha_1}{K+\alpha_1+\alpha_2}$.
\end{proposition}

The following corollary provides an interesting endnote to Theorem \ref{t:wk-exp-attr}.

\begin{corollary}[Corollary to Theorem \ref{t:wk-exp-attr}]
Let $\mathcal{M}_{-1,\alpha}$ be the weak exponential attractor admitted by the semiflow $S_\alpha$, $\alpha\in[0,1]$.
Then for each $\alpha\in[0,1]$, the set 
\[
\widetilde{\mathcal{M}}_{\alpha}:=\bigcup_{t>0}S_\alpha(t)\mathcal{M}_{-1,\alpha}.
\]
is positively invariant in $\mathcal{H}_0$. 
The set $\widetilde{\mathcal{M}}_{0}$ is bounded in $D(A^{1/2}_0)$, and when $\alpha\in(0,1]$, $\widetilde{\mathcal{M}}_{0}$ is bounded in $D(A_\alpha)$.
\end{corollary}

\begin{proof}
The asserted bounds follow directly from \eqref{regularity}.
The proof of the positive invariance is a straightforward calculation relying on the semigroup properties of $S_\alpha$; indeed, for all $\alpha\in[0,1],$
\begin{align}
S_\alpha(t)\widetilde{\mathcal{M}}_{\alpha} & = S_\alpha(t)\left( \bigcup_{s>0} S_\alpha(s)\mathcal{M}_{-1,\alpha} \right) = \bigcup_{s>0} S_\alpha(t+s)\mathcal{M}_{-1,\alpha}  \notag \\ 
& \subseteq \bigcup_{s>-t} S_\alpha(t+s)\mathcal{M}_{-1,\alpha} = \bigcup_{\tau>0} S_\alpha(\tau)\mathcal{M}_{-1,\alpha} = \widetilde{\mathcal{M}}_{\alpha}.
\end{align}
This completes the proof.
\end{proof}

\begin{remark}
Notice that we are not claiming that the sets $\widetilde{\mathcal{M}}_{\alpha}$ are compact, finite dimensional, nor (exponential) attractors.
\end{remark}

\section{Conclusions}  \label{s:conclusions}

In this article, we showed that the strongly damped perturbation of the weakly damped semilinear wave equation with hyperbolic dynamic boundary conditions possesses a semiflow generated by the mild solutions, in two different cases, the first corresponding to an abstract problem where the associated operator generates a $C_0$-semigroup of Gevrey class $\delta$ for $\delta>2$, with the other case corresponding to an analytic semigroup. 
Each case assumes the interior (bulk) potential has subcritical nonlinearity. 
After proving the existence of an absorbing set, we showed the associated semiflows admit a global attractor in the standard energy phase space. 
Indeed, we obtain the required asymptotic compactness through a suitable ``$\alpha$-contraction'' argument, which turns out to be rather necessary since other means known to evolution equation with dynamic boundary conditions cannot be applied here.
With that, we were able to establish certain decay and compactness properties over the decomposition of the difference of two solutions, at least in the weak topology.
The uniform global Lipschitz continuity of the semiflow on the absorbing set yielded the existence of weak exponential attractors, bounded in the standard topology, and compact in the weak topology.
The finite fractal dimension of the weak exponential attractors insures that the global attractors also possess finite fractal dimension, at least in the weak topology.
We also establish the upper-semicontinuity result for the family of global attractors $\mathcal{A}_\alpha$, as $\alpha\rightarrow 0,$ in the standard phase space $\mathcal{H}_0.$

It would be interesting (if possible) to extend the results presented here to the case of critical nonlinearities; e.g., replace the growth condition \eqref{assm-3} with the following,
\[
|f(r)-f(s)|\le\ell_1|r-s|(1+|r|^4+|s|^4).
\]
What keeps us from that presentation here is the lack of the local Lipschitz property for the nonlinear functional $\mathcal{F}$.
Recall, the mild solutions obtained here come through semigroup theory, and they are not necessarily the same as weak solutions (cf. Remark \ref{r:adjoint}). 
Since weak solutions would not necessarily require that $\mathcal{F}$ be locally Lipschitz on the weak energy phase space, one could feasibly employ an arbitrary (polynomial) growth condition, probably at the cost of introducing a new term, e.g. $\|u\|^p_{L^p(\Omega)}$, into the norm of the standard energy phase space. 
On the other hand, the difficulty with weak solution in this setting arises when attempting to obtain precompactness of the semiflow. 
Here, fractional powers of the Laplacian on $\Omega$ are not defined. 
Moreover, because of the dynamic flux component of the boundary condition, i.e. $\omega\partial_{\bf{n}}u_t$, there is no suitable $H^2$-regularity estimate that can be used to bound (part of) the solution operator in some smoother space.
Even for the strongly damped wave associated with homogeneous Dirichlet boundary conditions, when the power of the leading Dirichlet--Laplace operator equaled one, i.e., $\omega(-\Delta)^\theta$, $\theta=1$, the weak solutions obtained in \cite{Carvalho_Cholewa_02-2,Carvalho_Cholewa_02} followed after a nontrivial application of a nonlinear Alekseev's nonlinear variation of constants formula.

Theorem \ref{t:sg-theory} says something slightly more general than what we used here to prove that the semiflow was well-posed.
As previously mentioned, the nonlinearity represented by the functional $\mathcal{F}$ is locally Lipschitz on the energy space $\mathcal{H}_0$.
More generally, it need only be locally Lipschitz as a map from $D(A_\alpha^\theta) \to \mathcal{H}_0$ for $\theta \in [0,\frac{1}{\gamma})$, where $\gamma = 1$ if $\alpha > 0$ and $\gamma = 2$ if $\alpha = 0$.
For this reason, it would be interesting to achieve a full characterization of the domains of fractional powers of $A_\alpha$, cf.~Chen and Triggiani's result \cite{chen1990characterization}.
Indeed, we may offer up the following conjecture.
\begin{conjecture}
For $\alpha \in (0,1]$, we have
$$
D(A_\alpha^\theta) \subset H^1(\Omega) \times H^{\min\{1,2\theta\}}(\Omega) \times H^1(\Gamma) \times H^{\min\{1,2\theta\}}(\Gamma)
$$
for all $0 \leq \theta \leq 1$.\\
For $\alpha = 0$, we have
$$
D(A_0^\theta) \subset H^1(\Omega) \times H^{\min\{1,2\theta\}}(\Omega) \times H^1(\Gamma) \times H^{\min\{1/2,\theta\}}(\Gamma)
$$
for all $0 \leq \theta \leq 1$.
\end{conjecture}
The first statement of this conjecture may be obtained by analogy with \cite[Theorem 1.1]{chen1990characterization}.
In fact, one can show in the case where $\alpha = 1$ that $A_\alpha$ matches exactly the form of the operator appearing in \cite{chen1990characterization} (see also \cite{chen1989proof}) and therefore not only does the conjecture above hold, but the opposite containment also holds (so that the two sets are equal).
The other cases are considerably more complicated.
The statement for $\alpha = 0$ may be defended by analogy with the domain $D(A_0)$, in which the fourth component is in $H^{1/2}(\Gamma)$, a loss of ``one half derivative" as compared to the domain $D(A_\alpha)$ for $\alpha > 0$.
A proof of this conjecture, though interesting, would not, however, allow us to generalize growth condition \eqref{assm-1}, because the first component in the domain of the generator (and its fractional powers) is still only $H^1$.
Thus, if one is seeking to prove the well-posedness of the semi-flow using Theorem \ref{t:sg-theory}, then the results given in this paper are optimal. 

Another interesting direction for some future work could involve investigating the well-posedness, regularity, and asymptotic behavior corresponding to the case when $\omega=0$. 
In addition to bounding the (fractal) dimension of the attractors in $\mathcal{H}_0$, at least in the case $\omega=0,$ after further work, the continuity properties of the associated attractors may be sought; for example, the upper-semicontinuity of the global attractors at $\omega=0$ and the existence of a robust (H\"{o}lder continuous) family of exponential attractors for $\omega\in[0,1]$.
Based on Section \ref{s:global} and the work \cite[Section 4]{Frigeri&ShombergXX}, the upper-semicontinuity result may follow by assuming only $f,g\in C(\mathbb{R})$ satisfy (\ref{assm-1})--(\ref{assm-4}).
However, with additional regularity properties from the global attractors, the more traditional arguments used to show upper-semicontinuity of global attractors in \cite{Hale&Raugel88,Hale88} (also see \cite[Theorem 3.31]{Milani&Koksch05}) may also prove successful. 
On the other hand, one may inquire about applying a perturbation parameter $\ep$ to the inertial terms $u_{tt}$, then letting this approach zero.
Then, this of course is just a special case arising from a so-called memory relaxation of the Allen--Cahn equation equipped with dynamic boundary conditions.

\bigskip

\providecommand{\bysame}{\leavevmode\hbox to3em{\hrulefill}\thinspace}
\providecommand{\MR}{\relax\ifhmode\unskip\space\fi MR }
\providecommand{\MRhref}[2]{%
  \href{http://www.ams.org/mathscinet-getitem?mr=#1}{#2}
}
\providecommand{\href}[2]{#2}

\end{document}